\documentclass[12pt]{article}
\usepackage{amssymb,amsmath,theorem}

\newtheorem{theorem}{Theorem}
\newtheorem{lemma}[theorem]{Lemma}

\newtheorem{defn}[theorem]{Definition}
\newtheorem{remark}{\it Remark\/}

\theoremheaderfont{\scshape}

\newenvironment{proof}{\par\noindent\textsc{Proof.}}{\hfill$\square$\par}
\newcommand{\btl}[1]{\begin{theorem}\label{th:#1}}
\newcommand{\bt}{\begin{theorem}}
\newcommand{\etl}{\end{theorem}}
\newcommand{\eps}{\varepsilon}

\newcommand{\pa}{\partial}
\newcommand{\ep}{\varepsilon}
\newcommand{\RR}{{\mathbb{R}}}
\newcommand{\grad}{\mathop{{\rm grad}}}

\newcommand{\rest}[1]{\rule[-8pt]{.5pt}{14pt}_{\;#1}}
\newcommand{\bel}[1]{\begin{equation}\label{eq:#1}}
\newcommand{\eel}{\end{equation}}

\title{Schauder-type estimates and applications}%
\author{Satyanad Kichenassamy%
\thanks{Laboratoire de Math\a'ematiques (UMR 6056),
CNRS \&\ Univer\-sit\a'e de Reims Cham\-pagne-Ardenne, Moulin de la
Housse, B.P. 1039, F-51687 Reims Cedex 2, France. \emph{Email:
}{\tt satyanad.kichenassamy@univ-reims.fr} } }

\date{}%

\begin{document}

\maketitle

\begin{quote}
Appeared in: HANDBOOK OF DIFFERENTIAL EQUATIONS\\
Stationary Partial Differential Equations, volume 3\\
Edited by M. Chipot and P. Quittner
Elsevier, 2006, pp.~401--464.
\end{quote}

\tableofcontents
\vfill\eject

\begin{quote}{\small
{\bf Abstract.}  (Publisher summary.) The Schauder estimates are among the oldest and most useful tools in the modern theory of elliptic partial differential equations (PDEs). Their influence may be felt in practically all applications of the theory of elliptic boundary-value problems, that is, in fields such as nonlinear diffusion, potential theory, field theory or differential geometry and its applications. Schauder estimates give H\"older regularity estimates for solutions of elliptic problems with H\"older continuous data; they may be thought of as wide-ranging generalizations of estimates of derivatives of an analytic function in the interior of its domain of analyticity and play a role comparable to that of Cauchy's theory in function theory. They may be viewed as converses to the mean-value theorem: a bound on the solution gives a bound on its derivatives. Schauder theory has strongly contributed to the modern idea that solving a PDE is equivalent to obtaining an a priori bound that is, trying to estimate a solution before any solution has been constructed. The chapter presents the complete proofs of the most commonly used theorems used in actual applications of the estimates.}
\end{quote}

\section{Introduction}

The Schauder estimates are among the oldest and most useful tools
in the modern theory of elliptic partial differential equations
(PDEs). Their influence may be felt in practically all
applications of the theory of elliptic boundary-value problems,
that is, in fields such as nonlinear diffusion (in biology or
environmental sciences), potential theory, field theory, or
differential geometry and its applications.

Generally speaking, Schauder estimates give H\"older regularity
estimates for solutions of elliptic problems with H\"older
continuous data; they may be thought of as wide-ranging
generalizations of estimates of derivatives of an analytic
function in the interior of its domain of analyticity (Cauchy's
inequalities) and play a role comparable to that of Cauchy's
theory in function theory. They may be viewed as converses to the
mean-value theorem: a bound on the solution gives a bound on its
derivatives. The estimates generally become false if H\"older
continuity is replaced by mere continuity.

Schauder estimates have three aspects, corresponding to three
different ways of applying them:
\begin{itemize}
\item[(i)]
they are regularity results:
solutions with minimal regularity must be as regular as data
permit;

\item[(ii)]
they give the boundedness of the inverse of certain
elliptic operators;

\item[(iii)]
they give the compactness of these
inverses.
\end{itemize}

Schauder theory has strongly contributed to the modern
idea that solving a PDE is equivalent to obtaining an \emph{a
priori} bound, that is, trying to estimate a solution
\emph{before} one has constructed any solution.

We aim in the following pages to give the reader the means to make
use of the recent literature on the subject. We assume the reader
is familiar with the basic facts of Functional Analysis and
elliptic theory (see \cite{B}). For this reason, we give complete
proofs of the most commonly used theorems used in actual
applications of the estimates; we then survey the main
generalizations, with emphasis on recent work. General references
on Schauder estimates and their applications include
\cite{Aubin,CH,GT,han-lin,SK-fr,LU,LS,Morrey,N73,Shimakura,R,S,Sm}

\subsection{What are Schauder-type estimates?}

It is convenient to distinguish four kinds of estimates: interior,
weighted interior, boundary, and Fuchsian estimates. Each of them
is further divided into second-order and first-order estimates. We
begin with the second-order estimates.

The interior Schauder estimate expresses that, if $L$ is a
second-order elliptic operator $L$ with H\"older-continuous
coefficients,\footnote{See section 2 below for the definition of
the regularity classes used in this paper. Recall that an operator
$L=\sum_{ij}a^{ij}\pa_{ij}+b^i\pa_i+c$ is elliptic if the
quadratic form $a^{ij}\xi_i\xi_j$ is positive-definite.} the
$C^\alpha$ norm of any second-order derivative of $u$ on the ball
of radius $r$ is estimated by the sum of the $C^\alpha$ norm of
$Lu$ on the ball of radius $2r$, and the supremum of $u$ on the
same ball. It therefore contains the following information:
\begin{itemize}
\item[(i)]
$u$ is \emph{as smooth as the data allow}: even though $Lu$ is
just one particular combination of $u$ and its derivatives of
order two or less, the H\"older continuity of $Lu$ implies the
same regularity for all second-order derivatives;

\item[(ii)]
 the
regularity of $u$ is \emph{local}, in the sense that we require no
smoothness assumption on the value of $u$ on the boundary of the
ball of radius $2r$;

\item[(iii)]
the set of all functions $u$ such that
$\sup |u|$ and $\|Lu\|_{C^\alpha(|x|<2r)}$ are bounded by a fixed
constant $M$ is \emph{relatively compact} in the $C^{2}$ topology
of the ball of radius $r$.
\end{itemize}

The boundary Schauder estimate expresses that if $Lu=f$ on a
bounded domain of class $C^{2+\alpha}$ and if $u$ is equal on the
boundary to a function of class $C^{2+\alpha}$, then $u$ is of
class $C^{2+\alpha}$ up to the boundary.

The scaled, or weighted interior estimates, in their simplest
form, express that, if $Lu=f$ is $C^\alpha$, and $f$ is bounded,
then, as one tends to the boundary, (i) $\nabla u$ blows up at
most like $d^{-1}$, where $d$ is the distance to the boundary, and
$\nabla^2u$ like $d^{-2}$; (ii) the expression $|\nabla^2
u(P)-\nabla^2 u(Q)|/|P-Q|^\alpha$ is estimated by
$C(\min(d(P),d(Q))^{-2-\alpha}$.\footnote{Following common
practice, we use the "variable constant convention" according to
which the same letter $C$ is used to denote constants which may
change from line to line. It is consistent as long as (i) the
context makes clear on what quantities the constants depend and
(ii) one is not interested in the value of the constant, but only
in its existence. The convention may have been influenced by an
observation by Schauder to the effect that the best constants in
Schauder estimates are not well understood.} In particular, this
estimate does not imply that $d^2u$ is of class $C^{2+\alpha}$ up
to the boundary.

The Fuchsian estimates express that, in the above situation,
$d^2u$ is of class $C^{2+\alpha}$ up to the boundary provided that
(i) the $a^{ij}/d^2$ and $b^i/d$ satisfy a H\"older condition
near the boundary and
(ii) \emph{either} $L$ satisfies additional sign conditions on the
lower-order terms and $u$ is bounded, \emph{or} both $u$ and $f$
satisfy a flatness condition at the boundary. Condition (i) is
reminiscent of the scaling behavior of ODE of Fuchs-Frobenius
type, hence the terminology.

First-order estimates are similar, with the difference that they
give $C^{1+\alpha}$ regularity of the solution if $Lu$ is merely
bounded; the conditions on the coefficients of $L$ are also
slightly weaker than in the $C^{2+\alpha}$ case. First-order
estimates are often as useful as the second-order estimates, and
may generalize to nonlinear operators such as the $p$-Laplacian,
for which the $C^{2+\alpha}$ estimates are false.

\subsection{Why do we need Schauder estimates?}

Schauder estimates form the basis of very general existence
theorems, because the compactness information they contain makes
it possible to apply fixed-point theorems for compact operators
(see \cite{B,LS,N73,R,Sm}).

Schauder theory has many applications beyond existence theorems;
we mention: asymptotic behavior, at infinity or near
singularities; properties of eigenfunctions (Riesz-Fredholm theory
\cite{B} , Krein-Rutman theorem \cite{KR}); the method of sub- and
super-solutions for nonlinear problems, and bifurcation theory
(see \cite{N73,S,Sm}).

Schauder theory has not been rendered obsolete by the more recent
developments of Sobolev theory and variational methods, for the
following reasons:
\begin{itemize}
\item[(i)] Schauder theory applies to problems without
variational structure;

\item[(ii)] it produces existence results without
assuming uniqueness;

\item[(iii)] it is more convenient than Sobolev
theory in the sense that functions in $H^k$ are H\"older
continuous only for $k$ greater than the number of space
dimensions.
\end{itemize}

Of course, modern studies of nonlinear problems often
use Sobolev or de Giorgi-Nash theory to obtain a modicum of
regularity, and improve it using the Schauder estimates.

This survey is by no means an exhaustive report on regularity
theory; in particular, the de Giorgi-Nash theorem on H\"older
estimates for operators with bounded measurable coefficients, and
the literature it gave rise to, is not discussed. Special results
on particular equations such as the Monge-Amp\`ere equation, or
the Laplace equation on polyhedral domains, are only briefly
discussed. Regularity estimates for parabolic problems, including
probabilistic methods for diffusion processes \cite{SV}, fall
outside the scope of this volume devoted to stationary problems,
although many of the techniques are similar to those in the
stationary case.

\subsection{Why so many methods of proof?}

The wide variety of methods for the derivation of Schauder
estimates may be understood as follows: all proofs require the
following ingredients:
\begin{itemize}
  \item An estimate for a model problem
  (the Laplacian on the unit ball of $\RR^n$, a
  half-space, or a half-ball).
  \item A scaling argument which transfers estimates to balls of radius $R$.
  \item A linear change of coordinates which yields the result for
  constant-coefficient operators.
  \item A passage to continuous coefficients.
\end{itemize}
The second step is often formulated in terms of weighted norms
involving the distance to the boundary;
the third is achieved using Korn's device, which consists in
comparing the given operator with the operator with
coefficients ``frozen'' at one point; the fourth is streamlined
by the use of interpolation inequalities for (weighted) H\"older
norms. In fact, the first three steps follow from the invariance
properties of the Laplace operator; the variety of proofs
essentially comes from the different ways to exploit these
properties of the Laplace operator.

The period 1882-1934 has seen the emergence of \emph{derivations}
of estimates of which the definitive form was only gradually
discovered. In this first stage, one notices a tendency to try and
replace the estimation of Green's function by a direct estimation
of the solution. As a result, solution methods based on the
construction of integral operators which provide (approximate)
inverses developed separately from regularity theory for nonlinear
problems, and eventually gave rise to pseudo-differential
analysis.

The subsequent period, say from 1934 to 1964, was devoted to a
streamlining and elucidation of the methods, and culminated in the
generalization to systems \cite{ADN} of Schauder's estimates. At
the end of this period, estimates on Green's function had been evacuated
from the variable-coefficient case. They would re-appear
indirectly with the introduction of pseudo-differential inverses
of elliptic operators, but pseudo-differential techniques with
symbols of limited regularity are still not very well understood
\cite{taylor}.

Once the estimates had been discovered, it became possible to look for
\emph{verifications}: efficient ways to prove that the
estimates hold once they have been proved by other
methods. This search has brought about a change in perspective,
triggered by the needs of new applications: once the passage from
constant to variable coefficients had been streamlined, it became
clear that potential theory was still needed to prove the
estimates in the case of the Laplacian. In other words, all of the
refinements of Schauder theory were ultimately based on the direct
proof of the estimates for the Laplacian. And the various proofs
in this case ultimately make use of the invariance of the
Laplacian under translation, rotation and scaling.

Now, a problem in which a singularity occurs at a specific point
in space cannot be translation-invariant. The first step in
handling such problems would be to consider scaled Schauder
estimates in balls which become smaller as one approaches the
singularity; these ``blow-up'' methods lead naturally to weighted
Schauder estimates.\footnote{The question of
behavior at infinity is of a similar nature,
because infinity may be replaced by an isolated singularity by inversion.}
This approach does not yield optimal regularity. From 1990
onwards, the author showed that the correct regularity, first for
hyperbolic problems, and more recently, for elliptic problems, may
be understood by reducing the problem to a local model
of the typical form
\bel{p-fe} d^2\Delta u + \lambda d\nabla d\cdot\nabla u+\mu
u=f(d,P), \eel
where $d$ is the distance to the singularity locus;
$\lambda$ and $\mu$ are
usually constants.\footnote{In some cases, it is convenient to
allow them to be operators.} This leads to the Fuchsian estimates
mentioned above, which form part of a systematic technique
for finding the asymptotic behavior of solutions
(see \cite{SK-fr,SK-be,SK-ln} and sections 5 and 6.9).

Note that Fuchsian operators had been studied for their own sake
from the 1970s onwards, but the results obtained at this time were
slightly weaker than those required for application to nonlinear
problems. Of course, the idea of Fuchsian reduction is not to be
found in earlier work.

\subsection{Classification of proofs}

The various approaches to the estimates differ in their treatment
of the model case, and the characterization of H\"older continuity
they use. The modern theory is dominated by the fact that interior
$H^k$ estimates for harmonic functions are now considered more or
less obvious (see the beginning of the proof of theorem
\ref{th:cm2}).

A first proof is based on the direct estimation of the Newtonian
potential \cite{Holder,Hopf31,S31,S34a,S34b,DN,N53,GT}. A second
proof is based on the search for comparison functions, and
therefore uses only the maximum principle
\cite{Brandt66,Brandt69}. A third proof rests on the dyadic
decomposition of the Fourier transform of $u$ \cite{St}. A variant
may be based on a characterization of H\"older continuity
by mollification generalizing a property of the Poisson kernel \cite{T}.
A fourth proof rests on an integral
characterization of H\"older continuity
\cite{Campanato64,Campanato65,Giaquinta,Meyers,Morrey54}.
The regularity problem
for minimal surfaces has led to a fifth approach: consider scaled
versions of the graph of $u$ corresponding to smaller and smaller
scales and characterize their limit by a Liouville theorem
\cite{Simon}. A sixth proof consists in rescaling $u-P$ where $P$
is a second-degree polynomial \cite{caf-ams,cab-caf,han2}.

\subsection{Generalizations and variants}

The most important cases to which the second-order estimates on
bounded smooth domains may be generalized are: higher-order
equations of Agmon-Douglis-Nirenberg type, for which it is
possible to find a fundamental solution for a model problem with
constant coefficients \cite{ADN} and equations on unbounded
domains \cite{cc,pr}. Scaling interior estimates yields several,
non-equivalent results \cite{cc,SK-be,GH,Tr}. The first-order
estimates may hold under weaker conditions on the coefficients
\cite{Cordes59,Cordes61,caf-am}. It is also possible to obtain
estimates in cases when the r.h.s. is only H\"older with respect
to some of the variables \cite{fife}. Slightly stronger results
hold in two dimensions \cite[Ch.~12]{GT}. A simple example in
which the model problem is quite difficult is the case of the
Laplace equation on a polyhedral domain.\footnote{Separation of
variables shows that the smoothness of harmonic functions on a
wedge-like domain, with Dirichlet conditions, depends on the
opening angle of the wedge.}

Higher-order estimates may be obtained in the obvious manner, by
differentiating the equation, provided the nonlinearities are
smooth. The Schauder estimates are actually true for certain fully
nonlinear equations with non-smooth nonlinearities
\cite{BE,cab-caf,GT}.

There are cases in which the model case is not a linear,
constant-coefficient problem: for instance,
\begin{itemize}

\item[(i)] the
$p$-Laplacian---also invariant under a similar group---has the
property that solutions with right-hand side zero are not
necessarily of class $C^2$ (see \cite{evans,SK-th});

\item[(ii)] Fuchsian operators also admit non-smooth solutions with smooth data
\cite{GS,SK-be,SK-ln};

\item[(iii)] sub-elliptic operators, such as those
related to Carnot groups, are not close to the Laplacian either
\cite{capogna-han};

\item[(iv)] even the Laplacian on polyhedral
domains presents new features not found in the regularity theory
in smooth domains. All this led to a very recent surge of activity
on very simple models. Since the simplest non-trivial model beyond
the Laplacian is the Fuchsian case, we briefly explain how such
problems arise naturally.
\end{itemize}

When trying to generalize Schauder estimates to problems with
boundary degeneracy, we saw that the local model is not the
Laplacian any more: it is a problem with quadratic degeneracy of
special form; it is \emph{scale-invariant} but not
\emph{translation-invariant}. Let us mention a few further
contexts where such PDEs arise:
Axisymmetric potential theory leads to problems with singular
coefficients such as the (elliptic) Euler-Poisson-Darboux equation
\bel{epd} u_{rr}+\frac\lambda r u_r+u_{zz}=-4\pi\rho, \eel
where $\lambda$ is a constant. Many authors, especially Alexander
Weinstein (and Hadamard) stressed long ago the usefulness of this
equation and noted its remarkable behavior under transformations
of the form $u\mapsto r^\gamma u(r,z)$. It may be treated within
the framework of the general theory of degenerate
elliptic PDEs (Fichera), writing it in the form
\[ru_{rr}+\lambda  u_r+ru_{zz}=-4\pi r\rho.
\]
In this form, the problem is reminiscent of Legendre's equation,
which also admits a linear degeneracy (at $\pm 1$). Motivated by
the search for a higher-dimensional generalization of the
expansion into Legendre functions to several variables, a general
theory of the Dirichlet problem for elliptic equations with linear
degeneracy on the boundary was developed in the 1960s and 1970s.
The prototype of such problems is
\bel{p-fe1} d\Delta u + \lambda \nabla d\cdot\nabla
u=f(P),\quad\text{in } \Omega \eel
where $d$ is a smooth function of $P\in\Omega$, equivalent to the
distance to $\pa\Omega$ near the boundary. An analogue of Schauder
estimates may be derived by an explicit computation of Green's
function if $\lambda>0 $\cite{GS}; the essential step is the
analysis of a model problem on a half-plane, by Laplace transform
in the normal variable. This method does not seem to generalize to
the case of quadratic degeneracy.

These considerations took a new meaning when, in the 1990s, one
realized that nonlinear problems give rise, by a systematic
process of reduction (see \cite{SK-gaeta,SK-fr}), to problems
modeled upon the general Fuchsian-type problem
\bel{p-fe2} d^2\Delta u + \lambda d\grad d\cdot\grad u+\mu
u=f(P),\quad\text{in } \Omega.
\eel
Because of the quadratic degeneracy, the Laplace transform is not
helpful. Nevertheless, it is possible to analyze indirectly this
model problem (see \cite{SK-fr,SK-be,SK-ln} and sections 5 and
6.7).

\subsection{What process were the Schauder estimates discovered by?}

Many steps in the derivation of the Schauder estimates become
clearer if one recalls the historical development which led from
potential theory to the Schauder estimates. For this reason, we
give a historical sketch, starting from Poisson (1813).

\subsubsection{Does Poisson's equation hold?}

Consider the Newtonian potential in three dimensions:
\bel{i1} V(P)=\int_{\RR^3} \frac{\rho(Q)}{|P-Q|}dQ \eel
where $P\in\RR^3$ and $|P-Q|$ is the distance from $P$ to $Q$ and
integrals are extended over $\RR^3$. This integral represents, up
to a constant factor, the gravitational potential generated by the
mass density $\rho(Q)$, if $\rho\geq 0$. If $\rho$ takes
positive and negative values, it may be interpreted in terms of an
electrostatic potential. If the density is bounded and has limited
support, $V$ is defined by a convergent integral, and so is the
corresponding force field proportional to the gradient of $-V$,
formally given by
\[
-\nabla V(P)=\int \frac{Q-P}{|P-Q|^3}\rho(Q)dQ.
\]
If $P$ lies outside the support of $\rho$, the integral may be
differentiated again, to yield Laplace's equation
\[ \Delta V= 0,
\]
where $\Delta=\sum_{i=1}^3\pa_i^2$. However, if $\rho(P)\neq 0$,
differentiation of the force field yields a divergent integral,
because $1/|P-Q|^3$ is not integrable. Poisson (1813) showed that,
if $\rho$ is constant in the neighborhood of $P$, $V$ nevertheless
satisfies Poisson's equation at the point $P$:
\bel{p3} -\Delta V = 4\pi\rho, \eel
Indeed, one may split the density into two parts: a constant
density in a ball around $P$, and a density which vanishes in a
neighborhood of $P$. The first part yields a potential which may
be computed exactly: it is quadratic near $P$; the second yields a
solution of Laplace's equation. Gauss \cite{Gauss} then proved
that Poisson's equation is valid if the density is continuously
differentiable. After investigations by Riemann, Dirichlet and
Clausius, H\"older (1882) \cite{Holder} proved that the second
derivatives of the potential are continuous, and that Poisson's
equation holds, under the \emph{H\"older condition of order
$\alpha$}
\bel{hc} |\rho(P)-\rho(Q)|\leq C|P-Q|^\alpha, \eel
for some $\alpha\in(0,1)$. In fact, the second derivatives of $V$
also satisfy a H\"older condition. Furthermore, if $\rho$ is
merely continuous, the first-order derivatives of $V$ satisfy a
H\"older condition for any $\alpha$. The argument was streamlined
by Neumann \cite{Neumann}.\footnote{The continuity of $\rho$ is
not sufficient to ensure that $V$ is twice continuously
differentiable.
Necessary and sufficient conditions for the existence of second
derivatives were studied by Petrini.} This substantiates Poisson's
idea that the potential should be well-approximated by a quadratic
function near every point where $\rho$ is well-approximated by a
constant.

\subsubsection{Emergence of the Dirichlet problem}

At the same time it became clear that the Newtonian potential is
merely one among all possible solutions of Poisson's equation; in
fact, solutions may be parameterized by their values on the
boundary of sufficiently smooth bounded domains
$\Omega\subset\RR^3$: this leads us to \emph{Dirichlet problem}
\bel{dp} \left\{
\begin{array}{rll}
-\Delta V &= 4\pi\rho &\text{in }\Omega\\ V&= g &\text{on
}\pa\Omega
\end{array}
\right. \eel
It seemed at first sight that the Dirichlet problem should have a
unique solution on the grounds that it should represent the
equilibrium potential in $\Omega$ when the potential is prescribed
on the boundary and continuous. Dirichlet and Riemann worked on
the assumption that $V$ could be obtained by minimizing the
\emph{Dirichlet integral}
\bel{D-int} E[u,\Omega]=\int_\Omega |\nabla u(Q)|^2 dQ \eel
among all sufficiently regular $u$ which agree with $g$ on
$\pa\Omega$. Weierstrass pointed out that such an argument may
fail for certain variational principles, and it was only with the
advent of Hilbert spaces that a justification of this method could
be made, for smooth domains. But then, if we find a function $V$
which admits integrable first-order derivatives and minimizes
Dirichlet's integral, how do we know that it has second-order
derivatives and that it solves Poisson's equation? There is a
second difficulty:  the Dirichlet problem may have no continuous
solution if the boundary presents a sharp inward spike (``Lebesgue
spine''). Even for $\rho=0$, the Poincar\'e \emph{balayage}
method, re-formulated and simplified by Perron into the
\emph{method of sub- and supersolutions}, only proves that, for
continuous $g$, there is a unique solution which is continuous up
to the boundary if $\pa\Omega$ is well-behaved\footnote{For
instance, it is sufficient that $\pa\Omega$ satisfy an exterior
sphere condition. A necessary and sufficient condition is due to
Wiener.} but does not prove that the solution is smooth if the
data ($\Omega$, $\rho$ and $g$) are smooth.

The corresponding issues for equations with variable coefficients
and non-linearities also led to the need for regularity estimates:
the Calculus of Variations and Conformal Mapping led to nonlinear
elliptic equations such as the equation of minimal surfaces and
Liouville's equation ($\Delta u=e^u$) in two variables. Picard
emphasized the advantages of iterative methods for PDEs. Now, if
one wishes to solve iteratively an equation of the form
\[ \Delta u = f(u)
\]
to fix ideas, one should define a sequence of functions obtained
by solving the Poisson equations
\[ \Delta u_n=f(u_{n-1})
\]
with $n=1$, 2,\dots\ In view of the above results, it seems
appropriate to work in a space of functions the second derivatives
of which satisfy a H\"older condition. The first results in this
direction seem to be due to Bernstein (see \cite{BB}). The
\emph{continuity method} may be viewed as a outgrowth of these
efforts. But the iterative approach only allows one to reach
problems close to Poisson's equation. Other approaches, based on
the reduction to an integral equation on the boundary, required
detailed estimates on the Green's function for operators with
variable coefficients. In the course of this development,
estimates for second derivatives of solutions of PDEs with
variable coefficients in $n$ variables were obtained (Korn, E.
Hopf, Giraud, Kellogg, Schauder,\dots, see
\cite{Ca1,Giraud26,Giraud29,Kellogg,Lichtenstein,Hopf31}).

Schauder's approach is different: it reduces the problem to a new
fixed-point theorem: the Leray-Schauder theorem for compact
operators; the compactness is provided by estimates of second
derivatives in H\"older spaces.
Schauder's proof bypasses the construction of Green's
function for variable-coefficient operators, and opens the door to
the solution of wide classes of nonlinear equations.

\subsection{Outline of the paper}

Section 2 collects several characterizations of H\"older spaces,
and gives the main interpolation results which enable the passage
from constant to variable coefficients.

Section 3 illustrates the main proof techniques on the case of the
interior estimates for the Laplacian.

Section 4 deals with the passage from the model case (Laplacian on
a ball) to variable coefficients and general domains.

Section 5 gives the main general-purpose Fuchsian estimates.

Section 6 collects the most important applications; self-contained
proofs of the major topological tools are also included.

\section{H\"older spaces}

\subsection{First definitions}\label{sec:fd}

Let $\Omega\subset\RR^n$ be a domain (\emph{i.e.}, an open and
connected set).

\begin{defn}
A function $u$ is H\"older-continuous at the point $P$ of
$\Omega$, with exponent $\alpha\in(0,1)$, if
\[ [u]_{\alpha,\Omega,P}:=
\sup_{Q\in\Omega, Q\neq P}\frac{|u(P)-u(Q)|}{|P-Q|^\alpha}<\infty.
\]

It is H\"older-continuous over $\Omega$, or \emph{of class}
$C^\alpha(\Omega)$ if it satisfies this condition for every
$P\in\Omega$. We write $ [u]_{\alpha,\Omega}:=\sup_P
[u]_{\alpha,\Omega}$.

It is of class $C^\alpha(\Omega)$ if
\[ \|u\|_{C^\alpha(\Omega)}:=
\sup_\Omega |u| +[u]_{\alpha,\Omega}.
\]
\end{defn}

Functions of class $C^\alpha$ are in particular uniformly
continuous. It $\pa\Omega$ is smooth, one can extend $u$ by
continuity to a continuous function on $\overline\Omega$; for this
reason, it is sometimes convenient to write
$C^\alpha(\overline\Omega)$ for $C^\alpha(\Omega)$ in this case,
to emphasize that $u$ is continuous up to the boundary.

It is easy to check that
\[ [uv]_{\alpha,\Omega}\leq
\|u\|_{C^\alpha(\Omega)}\|v\|_{C^\alpha(\Omega)}.
\]

Higher-order H\"older spaces $C^{k+\alpha}(\Omega)$ are defined in
the natural way: first, write $|\nabla^ku|$ for the sum of the
absolute values of the derivatives of $u$ of order $k$, and define
$[\nabla^ku]_{\alpha,\Omega}$ similarly. Let
\[\|u\|_{C^k(\Omega)}:= \max_{0\leq j\leq k}\sup_\Omega |\nabla^ju|,
\]
and
\[ \|u\|_{C^{k+\alpha}(\Omega)}:=
\|u\|_{C^k(\Omega)} +[\nabla^k u]_{\alpha,\Omega}.
\]
In all these norms, the reference domain $\Omega$ will be omitted
whenever it is clear from the context.

\subsection{Dyadic decomposition}

The H\"older spaces defined above are all Banach spaces, but
smooth functions are not dense in them: even in one dimension, if
$(f_m)$ is a sequence of smooth functions and $f\in C^\alpha(\RR)$
is such that $\|f-f_m\|_{C^\alpha(\RR)}\to 0$, one proves easily
that for any $P$ and any $\ep>0$, there is a neighborhood of $P$
on which $|f(P)-f(Q)|\leq \ep|P-Q|^\alpha$. In other words,
$\lim_{Q\to P}|f(P)-f(Q)||P-Q|^{-\alpha}=0$. Any function $f$
which does not satisfy this property cannot be approximated by
smooth functions in the $C^\alpha$ norm.

Nevertheless, there is a systematic way to decompose
H\"older-continuous functions on $\RR^n$ into a uniformly
convergent sum of smooth functions: define the Fourier transform
of $u$ by
\[ \hat u(\xi) = \int_{\RR^n} e^{-ix\cdot\xi}u(x)\;dx
\]
and consider $\varphi\in C^\infty_0(\RR)$ such that
$0\leq\varphi\leq 1$, $\varphi=1$ for $|x|\leq 1$, $\varphi=0$ for
$|x|\geq 0$. Define
\[ \hat u_0=\varphi(|\xi|)\hat u(\xi);\quad
   \hat u_j=[\varphi(2^{-j}|\xi|)-\varphi(2^{-(j-1)}|\xi|)]\hat
   u(\xi)\text{ for }j\geq 1.
\]
We let $\hat v_j=\hat u_0+\cdots+\hat u_j$.
\begin{defn} The decomposition
\[ u = \sum_{j\geq 0} u_j
\]
is the Littlewood-Paley (LP), or dyadic decomposition of $u$
\cite{St}.
\end{defn}
By Fourier inversion, we have
\[ u_j=\psi_j * u\text{ with }\psi_j(x)=2^{jn}\psi(2^jx),
\]
where $\psi(x)=(2\pi)^{-n}\int_{\RR^n}
[\varphi(|\xi|/2)-\varphi(|\xi|)]\exp(ix\cdot\xi)d\xi$.
Note that $\hat\psi$ vanishes near the origin; in
particular, $\hat\psi_j(0)=\int_{\RR^n}\psi_jdx=0$.
\btl{c-alph-char} Let $0<\alpha<1$.
\begin{enumerate}
  \item (Bernstein's inequality.)
  There is a constant $C$ such that,
  for any $k$, $\sup_x(|\nabla^k u_j|+|\nabla^k v_j|)\leq C 2^{jk}\sup_x |u(x)|$.
  \item If $u\in C^\alpha(\RR^n)$, there is a constant $C$
  independent of $j$ such that
          $\sup_x |u_j(x)| \leq C 2^{-j\alpha} \|u\|_{C^\alpha}$.
  \item Conversely, if the above inequality holds for every $j\geq
  1$, then $u\in C^\alpha(\RR^n)$.
\end{enumerate}
\etl
\begin{proof}
(1) On the one hand, we have $|u_j(x)|\leq \|\psi\|_{L^1}\sup |u|$
and $|v_j(x)|\leq \|\phi\|_{L^1}\sup |u|$. On the other hand, if
$a$ is a multi-index of length $k$,
\begin{align*}
  |\nabla^a u_j(x)| & =
      \left| \int u(y)2^{jk}\nabla^a\psi[2^j(x-y)]2^{jn}dy \right|\\
   & = C 2^{jk}\sup |u|.
\end{align*}
The result follows.

(2) Since $\int\psi(y)dy=0$, $u_j$ may be written, for $j\geq 1$,
\[u_j(x)=\int [u(x-y)-u(x)]2^{jn}\psi(2^jy)dy
=\int [u(x-z/2^j)-u(x)]\psi(z)dz.
\]
If $u\in C^\alpha$, it follows that
\[ |u_j(x)|\leq 2^{-j\alpha}[u]_\alpha\int |z|^\alpha |\psi(z)|dz,
\]
QED.

(3) Conversely, if the $u_j$ are of order $2^{-j\alpha}$, the
series $u_0+u_1+\cdots$ converges uniformly. Call its sum $u$; it
is readily seen that the $u_j$ do give its LP decomposition. We
may apply (1) to $u_{j-1}+u_j+u_{j+1}$, and obtain
\[ \sup_x |\nabla u_j(x)|\leq C 2^{j(1-\alpha)}.
\]
Writing $u=v_{j-1}+w_j$, where $w_j=u_j+u_{j+1}+\cdots$, we find
that
\begin{align*}
|u(x)-u(y)| &\leq \sum_{j>k}|x-y|\sup|\nabla u_j|+2\sup|w_j| \\
            &\leq C|x-y|(1+\cdots+2^{(j-1)(1-\alpha)})+C2^{-j\alpha}\\
            &\leq C[2^{-j\alpha}+|x-y| 2^{j(1-\alpha)}].
\end{align*}
Choose $j$ such that $2^{-j}\leq|x-y|\leq 2^{-(j-1)}$. A bound on
$[u]_{\alpha}$ follows.
\end{proof}

\subsection{Weighted norms}

Several of the results we shall prove estimate the H\"older norm
of a function $u$ on a ball of radius $R$ in terms of bounds on
the ball of radius $2R$ with the same center. In order to exploit
these inequalities in a systematic fashion, it is useful to define
H\"older norms weighted by the distance to the boundary.

Let $\Omega\neq \RR^n$ and let $d(P)$ denote the distance from $P$
to $\pa\Omega$, and
\[ d_{P,Q}=\min(d(P),d(Q)).
\]
Let also $\delta$ be a smooth function in all of $\Omega$ which is
equivalent to $d$ for $d$ sufficiently small.\footnote{Such a
function is easy to construct if $\Omega$ is bounded and smooth.
Note that even in this case, $d$ is smooth only near the boundary;
see section 2.5.} Define, for $k=0$, $1,\dots$,
\[ \|u\|_{k,\Omega}^\#=\sum_{j=0}^k \sup_\Omega d^j|\nabla^j u|,
\]
and
\[ \|u\|_{k+\alpha,\Omega}^\#=\sum_{j=0}^k \|\delta^j
u\|_{C^{j+\alpha}(\Omega)},
\]
The spaces corresponding to these norms are called
$C^k_\#(\Omega)$, $C^{k+\alpha}_\#(\Omega)$. The space
$C^{k+\alpha}_*(\Omega)$ has the norm
\[ \|u\|_{k+\alpha,\Omega}^*=
\|u\|_{k,\Omega}^*+[u]_{k+\alpha,\Omega},
\]
where
\[ \|u\|^*_{k,\Omega}=\sum_{j=0}^k[u]_{j,\Omega}^*,
\]
with $[u]_{k,\Omega}^*=\sup_\Omega d^k|\nabla^k u|$, and
\[[u]_{k+\alpha,\Omega}^*=
\sup_{P,Q\in\Omega} d_{P,Q}^{k+\alpha} \frac{|\nabla^k
u(P)-\nabla^k u(Q)|}{|P-Q|^\alpha}.
\]
We also need the further definitions:
\[ [u]_{\alpha,\Omega}^{(\sigma)}=\sup_{P,Q\in\Omega} d_{P,Q}^{\alpha+\sigma}
\frac{|u(P)-u(Q)|}{|P-Q|^\alpha};
\qquad
\|u\|_{\alpha,\Omega}^{(\sigma)}=\sup_\Omega
|d^\sigma u|+[u]_{\alpha,\Omega}^{(\sigma)}.
\]

As before, the mention of $\Omega$ will be omitted whenever
possible.

\subsection{Interpolation inequalities}

\btl{int-l} For any $\ep>0$, there is a constant $C_\ep$ such that
\begin{align*}
[u]_1^* &\leq \ep [u]_2^* + C_\ep \sup |u|\\
[u]_1^* &\leq \ep [u]_{1+\alpha}^* + C_\ep \sup |u|\\
[u]_2^* &\leq \ep [u]_{2+\alpha}^* + C_\ep [u]_1^*\\
[u]_{1+\alpha}^*&\leq  \ep [u]_2^* + C_\ep \sup |u|.
\end{align*}
\etl
\begin{proof}
Recall the elementary inequality, for $C^2$ functions of one
variable $t\in[a,b]$,\footnote{For the proof, write
$f'(t)=f'(s)+\int_s^t f''(\tau)d\tau$, where $s$ satisfies
$f'(s)=(f(b)-f(a))/(b-a)$.}
\[ \sup |f'|\leq \frac 2{b-a}\sup |f|+(b-a)\sup |f^{\prime\prime}|.
\]
Fix $\theta\in (0,1/2)$, and $P\in\Omega$. Let $r=\theta d(P)$. If
$Q\in B_r(P)$, and $Z\in\pa\Omega$, we have
\[ |Z-Q| \geq |Z-P|-|P-Q|\geq d(P)(1-\theta)\geq \frac 12 d(P)\geq
r\geq |P-Q|.
\]
It follows in particular that $d(Q)\geq d(P)(1-\theta)\geq \frac12
d(P)$, hence
\[ d_{P,Q}\geq \frac 12 d(P).
\]
Applying the elementary inequality to $u$ restricted to the
segment $[P,P+re_i]$,\footnote{By the choice of $r$, this segment
lies entirely within $\Omega$.} where $e_i$ is the $i$-th basis
vector, we find
\[ |\pa_iu(P)|\leq \frac2r\sup_{B_r}|u|+r\sup_{B_r}|\pa_{ii}u|.
\]
It follows that
\[ \sup_{B_r}|\pa_{ii}u|\leq \sup d(Q)^{-2}\sup
d(Q)^2|\pa_{ii}u|\leq \frac{[u]_2^*}{d(P)^2(1-\theta)^2}.
\]
Therefore,
\[ [u]_1^*=\sup_{B_r}|d(Q)\pa_iu(Q)|\leq\frac
2\theta\sup|u|+\frac\theta{(1-\theta)^2}[u]_2^*.
\]
If we choose $\theta$ so that $\theta(1-\theta)^{-2}\leq\ep$, we
arrive at the first of the desired inequalities.

For the second inequality, we note that, using again the
mean-value theorem, there is on the segment  $[P,P+re_i]$ some
$\tilde P$ such that $|\pa_i u(\tilde P)|\leq (2/r)\sup_{B_r}|u|$.
It follows that
\begin{align*}
|\pa_iu(P)| &\leq |\pa_iu(\tilde P)|+|\pa_iu(P)-\pa_iu(\tilde P)|
\\
            &\leq \frac 2r \sup_\Omega |u| \\
            &\quad +
                  (\sup_{Q\in B_r(P)} d_{P,Q}^{-1-\alpha})
                  |P-\tilde P|^\alpha
                  \sup_{Q\in B_r(P)} d_{P,Q}^{1+\alpha}
                  \frac{|\nabla u(P)-\nabla u(Q)|}{|P-Q|^\alpha}\\
            &\leq \frac 2r \sup_\Omega |u| +
                  (2/d(P))^{1+\alpha}(\theta d(P))^\alpha
                  [u]_{1+\alpha}^*.
\end{align*}
Multiplying through by $d(P)=r/\theta$, we find the second
inequality.

A similar argument gives the third and fourth inequalities.
\end{proof}

\subsection{Properties of the distance function}
\label{sec:prel}

We prove a few properties of the function $d(x)=\mathop{\rm
dist}(x,\pa\Omega)$, when $\Omega$ is bounded with boundary of
class $C^{2+\alpha}$. Without smoothness assumption on the
boundary, all we can say is that $d$ is Lipschitz; indeed, since
the boundary is compact, there is, for every $x$ a $z\in\pa\Omega$
such that $d(x)=|x-z|$. If $y$ is any other point in $\Omega$, we
have $d(y)\leq |y-z|\leq |y-x|+|x-z|=|y-x|+d(x)$. It follows that
$|d(x)-d(y)|\leq |x-y|$. For more regular $\pa\Omega$, we have the
following results:
\btl{5} If $\pa\Omega$ is bounded of class
$C^{2+\alpha}$,
\begin{enumerate}
  \item there is a $\delta>0$ such that every point such that
  $d(x)<\delta$ has a unique nearest point on the boundary;
  \item in this domain, $d$ is of class $C^{2+\alpha}$;
  furthermore, $|\nabla d|=1$, and
\[-\Delta d=\sum_j \frac{\kappa_j}{1-\kappa_jd},
\]
where $\kappa_1,\dots,\kappa_{n-1}$ are the principal curvatures
of $\pa\Omega$.
In particular, $-\Delta d/(n-1)$ is equal to the mean curvature of
the boundary.
\end{enumerate}
\etl
\begin{proof}
We work near the origin, which we may take on $\pa\Omega$. Our
proofs will give local information near the origin, which can be
made global by a standard compactness argument.

Choose the coordinate axes so that $\Omega$ is locally represented
$\{x_n>h(x')\}$, where $x'=(x_1,\dots,x_n)$ and $h$ is of class
$C^{2+\alpha}$ with $h(0)=0$ and $\nabla h(0)=0$. We may also
assume that the axes are rotated so that the Hessian
$(\pa_{ij}h(0))$ is diagonal. Its eigenvalues are, by definition,
the principal curvatures $\kappa_1,\dots,\kappa_{n-1}$ of the
boundary. Their average is, again by definition, the mean
curvature of the boundary.

At any boundary point, the vector with components
\[(\nu_i)=(-\pa_1 h,\dots,-\pa_{n-1}h,1)/\sqrt{1+|\nabla h|^2}
\] is the \emph{inward}
normal to $\pa\Omega$ at that point. One checks
$\pa_j\nu_i(0)=-\pa_{ij}h(0)=\kappa_j\delta_{ij}$ for $i$ and $j$
less than $n$. Thus, $\nu$ is of class $C^1$. For any $T>0$ and
$y\in\RR^{n-1}$, both small, consider the point
$x(Y,T)=(Y,h(Y))+T\nu(Y)$; this represents the point obtained by
traveling the distance $T$ into $\Omega$, starting from the
boundary point $(Y,h(Y))$, and traveling along the normal. We
write
\[ \Phi : (Y,T)\mapsto x(Y,T).
\]
We want to prove that all points in a neighborhood of the boundary
are obtained by this process, in a unique manner: in other words,
$(Y,h(Y))$ is the unique closest point from $x(Y,T)$ on the
boundary, provided that $T$ is positive and small. It suffices to
argue for $Y=0$; in that case, since $h$ is $C^2$, it is bounded
below by an expression of the form $a|Y|^2$, which implies that
for $T$ sufficiently small, the sphere of radius $T$ about
$x(Y,T)$ contains no point of the boundary except the
origin.\footnote{Indeed, the equation of this sphere is
$x_n=T-\sqrt{T^2-|Y|^2}$, which, by inspection, is bounded below
by $a|Y|^2$ for $2aT< 1$.} We may now consider the new coordinate
system $(Y,T)$ thus defined. We compute, for $Y=0$, but $T$ not
necessarily zero,
\[ \frac{\pa x_i}{\pa Y_j}=\delta_{ij}(1-\kappa_j T)
\]
for $i$ and $j<n$, while
\[ \frac{\pa x_n}{\pa Y_j}=\frac{\pa x_i}{\pa T}=0;\quad
\frac{\pa x_n}{\pa T}=1.
\]
The inverse function theorem shows that, near the origin, the map
$\Phi$ and its inverse are of class $C^1$, and that the Jacobian
of $\Phi^{-1}$ is, for $Y=0$,
\[\frac{\pa(Y,T)}{\pa x}=
\mathop{\rm diag}(\frac 1{1-\kappa_1 T},\dots,\frac
1{1-\kappa_{n-1} T},1).
\]
In fact, $\Phi^{-1}$ is of class $C^{1+\alpha}$. Indeed, $\Phi$
has this regularity, and the differential of $\Phi^{-1}$ is given
by $[\Phi'\circ\Phi^{-1}]^{-1}$, and the map $A\mapsto A^{-1}$ on
invertible matrices is a smooth map. Since $\nu(Y)$, which is
equal to the gradient of $d$, is a $C^{1+\alpha}$ function of $Y$,
we see that it is also a $C^{1+\alpha}$ function of the $x$
coordinates. It follows that $d$ is of class $C^{2+\alpha}$. The
computation of the second derivatives of $d$ is now a consequence
of the computation of the first-order derivatives of $\nu$.

It follows from this discussion that $T=d$ near the boundary, and
that $|\nabla d|=1$; in fact $\nabla d=\nu$.
\end{proof}

\subsection{Integral characterization of H\"older continuity}

Let $\Omega$ be a bounded domain. Write $\Omega(x,r)$ for
$\Omega\cap B(x,r)$. We assume that the measure of $\Omega(x,r)$
is at least $Ar^n$ for some positive constant $A$, if $x\in\Omega$
and $r\leq 1$. This condition is easily verified if $\Omega$ has a
smooth boundary. Define the average of $u$:
\[ u_{x,r}=|\Omega(x,r)|^{-1}\int_{\Omega(x,r)}u\,dx.
\]
\btl{char-c-al} The space $C^\alpha(\Omega)$ coincides with the
space of (classes of) measurable functions which satisfy
\[ \int_{\Omega(x,r)}|u(y)-u_{x,r}|^2\,dy\leq Cr^{n+2\alpha}
\]
for $0<r<\mathop{\rm diam}\Omega$. The smallest constant $C$,
denoted by $\|u\|_{\mathcal L^{2,n+2\alpha}}$ is equivalent to the
$C^\alpha(\Omega)$ norm. \etl
\begin{remark}
If one defines $\mathcal L^{p,\lambda}$ by the property:
$\int_{\Omega(x,r)}|u(y)-u_{x,r}|^p\,dy\leq Cr^{\lambda}$, with
$n<\lambda < n+p$, one obtains a characterization of the space
$C^{(\lambda-n)/p}$.
\end{remark}
\begin{proof}
The integral estimate is clearly true for H\"older continuous
functions. Let us therefore focus on the converse. We first prove that
$u$ is uniformly approximated by its averages, and then derive a modulus
of continuity for $u$.

If $x_0\in\Omega$ and $0<\rho<r\leq 1$, we have
\begin{align*}
A\rho^n |u_{x_0,\rho}-u_{x_0,r}|^2
&\leq \int_{\Omega(x_0,\rho)}|u_{x_0,\rho}-u_{x_0,r}|^2dx\\
&\leq 2\left(
\int_{\Omega(x_0,\rho)}|u-u_{x_0,\rho}|^2dx +
\int_{\Omega(x_0,r)}|u-u_{x_0,r}|^2dx
       \right)\\
&\leq C(r^\lambda+\rho^\lambda).
\end{align*}
Letting $r_j=r2^{-j}$ and $u_j=u_{x_0,\rho_j}$ for $j\geq 0$,
we find
\[ |u_{j+1}-u_j|\leq C 2^{j(n-\lambda)/2}r^{(\lambda-n)/2}
                  =  C2^{-j\alpha}r^\alpha.
\]
For almost every $x_0$, the Lebesgue differentiation theorem
ensures that $u_j\to u(x_0)$ as $j\to\infty$. It follows that
\[ |u(x_0)-u_{x_0,r}|\leq \sum_j|u_{j+1}-u_j|\leq Cr^\alpha.
\]
Since $u_{x,r}$ is continuous in $x$ and converges uniformly as
$r\to 0$, it follows that $u$ may be identified, after
modification on a null set, with a continuous function.

To estimate its modulus of continuity, we need the following
result:
\begin{lemma}
Let $u\in \mathcal L^{2,n+2\alpha}$, $x$, $y$ two points in
$\Omega$, and $r=|x-y|$; we have
\[ |u_{x,r}-u_{y,r}|\leq C r^\alpha.
\]
\end{lemma}
\begin{proof}
We may assume $r=|x-y|\leq 1$. If $z\in
B_r(x)$, we have $|z-y|\leq r+|x-y|\leq 2r$. Therefore
$\Omega(y,2r)\supset\Omega(x,r)$. It follows that
$\Omega(x,2r)\cap\Omega(y,2r)\supset\Omega(x,r)$ has measure $Ar^n$
at least. We therefore have
\begin{align*}
|\Omega(x,2r)\cap\Omega(y,2r)|&|u_{x,2r}-u_{y,2r}| \\
  &\leq
  \int_{\Omega(x,2r)}|u(z)-u_{x,2r}|dz+
  \int_{\Omega(y,2r)}|u(z)-u_{y,2r}|dz \\
  &\leq
  \left[\int_{\Omega(x,2r)}|u(z)-u_{x,2r}|^2dz\right]^{1/2}|\Omega(x,2r)|^{1/2}\\
  &\qquad + \left[\int_{\Omega(y,2r)}|u(z)-u_{y,2r}|^2dz\right]^{1/2}|\Omega(y,2r)|^{1/2}\\
  &\leq Cr^{\alpha+n/2}r^{n/2}.
\end{align*}
It follows that
\[ |u_{x,2r}-u_{y,2r}|\leq C A^{-1}r^\alpha,
\]
\end{proof}
To conclude the proof of the theorem, it suffices to estimate
$|u(x)-u(y)|$ by $|u(x)-u_{x,r}|+|u_{x,r}-u_{y,r}|+|u_{y,r}-u(y)|
\leq 2Cr^\alpha +|u_{x,r}-u_{y,r}|$.
\end{proof}

\section{Interior estimates for the Laplacian}
\subsection{Direct arguments from potential theory}

Let $n\geq 2$, and let $B_R(P)$ denote the open ball of radius $R$
about $P$. Mention of the point $P$ is omitted whenever this does
not create confusion. The volume of $B_R$ is $\omega_nR^n$ and its
surface $n\omega_n R^{n-1}$. The Newtonian potential in $n$
dimensions is
\[ g(P,Q)=\frac{|P-Q|^{2-n}}{(2-n)n\omega_n}\text{ for }n\geq 3
\]
and
\[ \frac 1{2\pi}\ln |P-Q|\text{ for }n=2.
\]
It is helpful to note that
\begin{enumerate}
  \item The derivatives of $g$ of order $k\geq 1$ w.r.t.\ $P$ are $O(|P-Q|^{2-n-k})$.
  \item The average of each of these second derivatives over the
  sphere $\{Q : |P-Q|=\text{const.}\}$ vanishes.\footnote{To
  check this, it is useful to note that the average of $x_i^2/r^2$
  over the unit sphere $\{r=1\}$ is equal to $\frac1n$, and
  similarly, using symmetry, the average of $(x_i-y_i)(x_j-y_j)/|x-y|^2$
  over the set $\{|x-y|=\text{const.}\}$ vanishes for $i\neq j$.}
\end{enumerate}
Next, consider, for $f\in L^1\cap L^\infty(\RR^n)$, the integral
\[
u(P)=\int_{\RR^n} g(P,Q)f(Q)\,dQ.
\]
We wish to estimate $u$ and its derivatives in terms of bounds on
$f$. Because of the behavior of $g$ as $P\to Q$, $g$ and its first
derivatives are locally integrable, but its second derivative is not.

It is easy to see that, if the point $P$ lies outside the
support of $f$, $u$ is smooth near $P$ and satisfies
$\Delta u=0$. For this reason, it suffices to study the
case in which the density $f$ is supported in a neighborhood of
$P$.

We prove in the next three theorems: (i) a pointwise bound on $u$
and its first-order derivatives; (ii) a representation of the
second-order derivatives which involves only locally integrable
functions; (iii) a direct estimation of
$\nabla^2u(P)-\nabla^2u(Q)$ using this representation.

\bt If $f$ vanishes outside $B_R(0)$, we have
\[ \sup_{B_R}(|u|+|\nabla u|)\leq C R^2\sup f,
\]
and $\nabla u$ is given by formally differentiating the integral
defining $u$. \etl
\begin{proof}
Consider a cut-off function
$\varphi_{\ep}(P,Q):=\varphi(|P-Q|/\ep)$, where $\varphi(t)$ is
smooth, takes its values between 0 and 1, vanishes for $t\leq 1$
and equals 1 for $t\geq2$. Considering the functions
\[ u_\ep(P)=\int g(P,Q)\varphi_\ep(P,Q)f(Q)dQ,
\]
which are smooth, it is easy to see that the $\pa_iu_\ep$ converge
uniformly, as $\ep\downarrow 0$, to $\int\pa_i g(P,Q)f(Q)dQ$.
Similarly, $u_\ep$ converges to $u$. Therefore, $u$ is
continuously differentiable. Using the growth properties of $g$
and its derivatives, we may estimate $\pa_iu(P)$ by
\[C\int_{B_{2R}(P)}C|P-Q|^{1-n}\sup|f|dQ,
\] because $B_R(0)\subset
B_{2R}(P)$. Taking polar coordinates centered at $P$, the result
follows.
\end{proof}
\vskip 1em
The case of second derivatives is more delicate, since the second
derivatives of $g$ are not locally integrable. We know since
Poisson that the integral defining $u$ is smooth near $P$ if $f$
is constant in a neighborhood of $P$. This suggests a reduction to
the case in which $f$ vanishes at $P$. We therefore first prove,
for such $f$, a representation of the second-order
derivatives which circumvents the fact that the second-order
derivatives of $g$ are not integrable.
\btl{wij} If $f$ has support in a bounded neighborhood $\Omega$ of
the origin, with smooth boundary, and if $f\in C^\alpha(\RR^n)$
for some $\alpha\in(0,1)$, then all second-order derivatives of
$u$ exist, and are equal to
\[ w_{ij}:=\int_\Omega \pa_{ij}g(P,Q)[f(Q)-f(P)]dQ
-f(P)\int_{\pa\Omega}\pa_i g\, n_j ds(Q),
\]
where derivatives of $g$ are taken with respect to its first
argument, and $n_j$ are the components of the outward normal to
$\pa\Omega$.
\etl
\begin{proof}
To establish the existence of second derivatives, we consider
\[ v_{i\ep}(P)=\int \pa_i g(P,Q)\varphi_\ep(P,Q)f(Q)dQ,
\]
which converges pointwise to $\pa_iu(P)$; in fact, since
$1-\varphi_\ep$ is supported by a ball of radius $2\ep$, a direct
computation yields $|u_i-v_{i\ep}|(P)=O(\ep\sup|f|)$. Now, writing
$P=(x_i)$ and $Q=(y_i)$, we have
\begin{align*}
\pa_jv_{i\ep}(P)=\int_\Omega\pa_{x_j}&(\varphi_\ep\pa_{x_i}g)(P,Q)[f(Q)-f(P)]dQ\\
&+f(P)\int_\Omega(\varphi_\ep\pa_{x_i}g)(P,Q)dQ.
\end{align*}
Now, since $\varphi_\ep$ and $g$ only depend on $|P-Q|$, we may
replace $\pa/\pa x_j$ by $-\pa/\pa y_j$ and integrate by parts.
This yields
\begin{align*}
\pa_jv_{i\ep}(P)=\int_\Omega\pa_{x_j}&(\varphi_\ep\pa_{x_i}g)(P,Q)[f(Q)-f(P)]dQ\\
&-f(P)\int_{\pa\Omega}\varphi_\ep\pa_{x_i}g(P,Q)n_j(Q)ds(Q).
\end{align*}
We may now estimate the difference $\pa_jv_{i\ep}-w_{ij}$ using
the same method as for the first-order derivatives. It follows
that $\pa_{ij}u=w_{ij}$.
\end{proof}
\vskip 1em
We now give the main estimate for second-order derivatives.
\bt Let
\[u(P)=\int_{B_{2R}(0)} g(P,Q)f(Q)\,dQ,
\]
where $f\in C^\alpha(B_{2R})$, with $0<\alpha<1$. Then
\begin{equation}
\sup_{B_R}|\nabla^2u|+[\nabla^2u]_{\alpha, B_R} \leq
C(\sup_{B_{2R}}|f|+R^\alpha[f]_{\alpha,B_{2R}}).
\end{equation}
\etl
\begin{proof}
We wish to estimate the regularity of $\pa_{ij}u$; we therefore
study $|\pa_{ij}u(P)-\pa_{ij}u(P')|$, for $P$, $P'$ in $B_R(0)$,
where the second derivatives are given by the expressions in
the previous theorem. The main step is to decompose the first
integrand in the resulting expression for $w_{ij}(P)-w_{ij}(P')$
into
\[[f(Q)-f(P')][\pa_{ij}g(P,Q)-\pa_{ij}g(P',Q)]+[f(P')-f(P)]\pa_{ij}g(P,Q).
\]
We therefore need to estimate the following quantities
\begin{enumerate}
  \item[(I)] $f(P)[\pa_ig(P,Q)-\pa_ig(P',Q)]$ for $Q\in\pa
  B_{2R}$.
  \item [(II)] $[f(P)-f(P')]\pa_ig(P',Q)$ for $Q\in\pa
  B_{2R}$.
  \item [(III)] $[f(P')-f(P)]\pa_{ij}g(P,Q)$ for $Q\in
  B_{2R}$.
  \item [(IV)] $[f(Q)-f(P')][\pa_{ij}g(P,Q)-\pa_{ij}g(P',Q)]$ for
  $Q\in B_{2R}$.
\end{enumerate}
The first boundary term (I) is easy to estimate using the
mean-value theorem:
\[ |\pa_ig(P,Q)-\pa_ig(P',Q)|\leq
|P-P'|\sup_{\xi\in[P,P']}|\nabla \pa_ig(\xi,Q)|.
\]
Now, since $Q\in \pa B_{2R}$, and $\xi\in B_R$, we have $|\xi-
Q|\geq 2R-R=R$, hence the supremum in the above formula is bounded
by a multiple of $R^{-n}$. Integrating, we get a contribution
$O(|P-P'|/R)$, which is \emph{a fortiori}
$O(|P-P'|^\alpha/R^\alpha)$.

Expression (II) is $O(|P-P'|^\alpha)$ since $f$ is of class
$C^\alpha$.

To estimate (III) and (IV), let $r_0=|P-P'|$ and $M$ be the
midpoint of $[P,P']$. We distinguish two cases: (i) When
$|Q-M|>r_0$, the distance from $Q$ to any point on the segment
$[P,P']$ is comparable to $|Q-M|$; this will enable a direct
estimation of (IV) using the mean-value theorem, and of (III) by
integration by parts. (ii) On the set on which
$|Q-M|\leq r_0$, we may directly estimate the sum of (III) and
(IV); the smallness of the region of integration compensates the
singularity of the derivatives of $g$.

We begin with the first case:
consider first the integral of (III) over the set
\[ A:=\{Q\in B_{2R} :  |Q-M|>r_0.\}.
\]
Its boundary is included in $\pa B_{2R}(0)\cup\pa B_{r_0}(M)$.
Integrating by parts and using the fact that, on
this set, $|P-Q|$ is bounded below by $\min(R,r_0/2)$, we find
that (III) $=O(|P-P'|^\alpha)$. For the term (IV), integrated over
the same set, we estimate $\pa_{ij}g(P,Q)-\pa_{ij}g(P',Q)$ by
$C|P-P'||\xi-Q|^{-n-1}$, for some $\xi\in[P,P']$. Using the
H\"older continuity of  $f$, the integral of (IV) is estimated by
\[ Cr_0\frac{|Q-P'|^\alpha}{|Q-\xi|^{n+1}}.
\]
Its integral over $A$ is estimated by its integral over
\[ A':=\{Q : |Q-M|>r_0.\}.
\]
On $A'$,
\[ |Q-P'|\leq |Q-M|+|M-P'|
           =  |Q-M|+\frac 12 r_0
         \leq \frac 32 |Q-M|.
\]
On the other hand,
\[ |Q-\xi|\geq |Q-M|-|M-\xi|
          \geq  |Q-M|-\frac 12 r_0
          \geq \frac 12 |Q-M|.
\]
Combining the two pieces of information, we find
\begin{align*}
\int_{A'} Cr_0\frac{|Q-P'|^\alpha}{|Q-\xi|^{n+1}}dQ
&\leq Cr_0\int_{A'} |Q-M|^{\alpha-n-1}dQ\\
&=Cr_0\int_{r_0}^\infty r^{\alpha-2}dr
 = C|P-P'|^\alpha.
\end{align*}
This completes the analysis of the integrals of (III) and (IV)
over $A$.

It remains to consider (III) and (IV) over the part of $B_{2R}$ on
which $|Q-M|\leq r_0$. In this case, $|P-Q|\leq
|P-M|+|M-Q|\leq\frac 32 r_0$, and similarly for $|P'-Q|$. We
therefore estimate directly the sum of (III) and (IV), namely
\[[f(Q)-f(P)]\pa_{ij}g(P,Q)-[f(Q)-f(P')]\pa_{ij}g(P',Q),
\]
by
\begin{align*} C[f]_{\alpha,B_{2R}}
\int_{|Q-M|<r_0}&(|Q-P|^{\alpha-n}+|Q-P'|^{\alpha-n})dQ \\ &\leq
C[f]_{\alpha,B_{2R}} \int_0^{3r_0/2}|Q-P|^{\alpha-1}d|Q-P|\leq
Cr_0^\alpha.
\end{align*}
Since $r_0=|P-P'|$, this completes the proof.
\end{proof}

\subsection{$C^{1+\alpha}$ estimates via the maximum principle}

We give two estimates for function such that $\Delta u$ is
bounded. The result is essentially optimal, and relies only on the
maximum principle. The choice of comparison functions is motivated
by numerical approximations for second-order derivatives; in this
sense, the argument may be compared with Nirenberg's
``method of translations'' for the proof of $L^2$-type estimates.
Second-order estimates may also be derived by
this method, but the choice of comparison functions is much more
complicated.

We begin with the $C^1$ estimate.
\btl{c1-mp} If $\Delta u= f$ on $K=\{ |x_i|<s\text{ for
}i=1,\dots, n \}$, then
\[ |\pa_nu(0)|\leq \frac ns \sup_K |u| +\frac d2 \sup_K |f|.
\]
\etl
\begin{proof}
Let $M=\sup_K |u|$, $N=\sup_K |f|$,
\[ v(x',x_n)=\frac 12 [u(x',x_n)-u(x',-x_n)]
\]
and
\[ w(x',x_n)=\frac M{s^2}[|x'|^2+x_n(ns-(n-1)x_n)]+\frac 12 Nx_n(s-x_n).
\]
Applying the maximum principle to $w\pm v$ on $K\cap\{0<x_n<s\}$
we obtain
\[ \frac 1{2x_n}|u(x',x_n)-u(x',-x_n)|\leq \frac
M{s^2}(ns-(n-1)x_n)+\frac N2(s-x_n).
\]
Letting $x_n\to 0$, one finds the desired inequality.
\end{proof}
\vskip 1em
We now turn to the continuity of the gradient of $u$. The result
implies interior $C^{1+\alpha}$ regularity for every $\alpha<1$.
\bt Let $\mu=\sup_K|\nabla u|$. There is a constant $k$ such that:
\[ \frac 12 |\pa_i u(0,x_n)-\pa_i u(0,-x_n)|\leq \mu \frac {x_n}s +
kx_n\ln\frac {x_n}s
\]
for $|x_n|\leq s/4$. \etl
\begin{proof}
It suffices to prove the result for $i=n$ and $i=n-1$.

For the case $i=n$, we consider the function of $n+1$ variables
$(x',y,z)$ defined by
\[ \phi(x',y,z)=\frac 14
[u(x',y+z)-u(x',y-z)-u(x',-y+z)+u(x',-y-z)],
\]
and the operator $L=\sum_{i<n}\pa^2_i+\frac12 (\pa_y^2+\pa_z^2)$,
so that one checks $|L\phi|\leq N$. We then compare $\phi$ with
\[ W = \frac {4M}syz+kyz\ln\frac{2s}{y+z},
\]
where
\[ k=\frac 43(N+\frac{8M}{s^2}(n-1)),
\]
on the set $K'=\{|x_i|<\frac s2\text{ if }i\leq n-1;\quad
0<y,z<\frac s4 \}$. Since
\[ LW=\frac{8M}{s^2}(n-1)+k\left[-1+\frac{yz}{(y+z)^2}\right]
\leq\frac{8M}{s^2}(n-1)-\frac 34 k=-N
\]
on $K'$, the maximum principle yields
\[ |\phi(0,y,z)|\leq y\left[\frac{4\mu}s+k\ln\frac{2s}{y+z}\right].
\]
Letting $z\to 0$ gives the first inequality in the theorem.

For the case $i=n-1$, we work with functions of $n$ variables
$(\tilde x,y,z)$, where $\tilde x=(x_1,\dots,x_{n-2})$, on the set
$K''=\{|x_i|<\frac s2\text{ if }i\leq n-2;\quad 0<y,z<\frac s2
\}$, and the auxiliary functions
\[\psi(\tilde x,y,z)=\frac 14
[u(\tilde x,y,z)-u(\tilde x,y,-z)-u(\tilde x,-y,z)+u(\tilde
x,-y,-z)],
\]
and
\[ \tilde W=\frac{4L}{s^2}|\tilde x|^2+yz\left[\frac{4\mu}s+\tilde
k\ln\frac{2s}{y+z}\right],
\]
where $\tilde k=\frac 23\left[N+\frac{8M}{s^2}(n-2)\right]$. One
finds $|\Delta\psi|\leq N$ and $-\Delta\tilde W\geq N$ on $K''$,
and the maximum principle yields the desired result as before.
\end{proof}

\subsection{$C^{2+\alpha}$ estimates via Littlewood-Paley theory}

LP decomposition provides a simple proof of the basic interior
estimate for the constant-coefficient case. This is essentially
due to the fact that the Fourier transform is rotation-invariant
and has a simple scaling behavior. The argument is however
tailored to isotropic situations. We give the argument in its
simplest form, with no aim at generality.
\btl{sch-lp}
Let $\rho\in C^\alpha(\RR^n)$ be such that $\hat\rho=O(|\xi|^5)$
near $\xi=0$. Then there is a $u\in C^{2+\alpha}(\RR^n)$ such that
$-\Delta u=\rho$. \etl
\begin{remark}
The condition $\hat\rho=O(|\xi|^5)$ means that the first few
moments of $\rho$ vanish; it may be achieved by subtracting from
$\rho$ a smooth potential with prescribed multipolar moments.
\end{remark}
\begin{proof}
Consider the LP decomposition $\rho_0+\rho_1+\cdots$ of $\rho$.
Define $u_j$ by $\hat u_j=\hat\rho_j/|\xi|^{-2}$. Then
$u=u_0+u_1+\cdots$ is well-defined, and the flatness assumption
ensures that $u_0$ and its first two derivatives are bounded. In
particular, $u_0$ is of class $C^{2+\alpha}$. Consider the Fourier
transform of $-\pa_{kl}u$:
\[ \xi_k\xi_l\hat u_j(\xi)
=\frac{\xi_k\xi_l}{|\xi|^{2}}\hat\rho(\xi)\hat\psi(2^j|\xi|)
=\hat\rho(\xi)\frac{(2^j\xi_k)(2^j\xi_l)\hat\psi(2^j|\xi|)}{(2^j|\xi|)^2}.
\]
Recall that $\hat\psi$ is flat at the origin, so that there is no
singularity for $\xi=0$.
Applying point (2) of theorem \ref{th:c-alph-char}
to the decomposition of $\rho$
in which $\psi$ would be replaced by $\psi'$, with
$\hat\psi'(\xi)=\xi_k\xi_l|\xi|^{-2}\hat\psi(|\xi|)$, we find
$\sup|(\pa_{kl}u)_j|=O(2^{-(2+\alpha)j})$. From the characterization of
$C^{2+\alpha}$ spaces, the result follows.
\end{proof}

\subsection{Variational approach}

We turn to a different approach, based on the integral
characterization of H\"older spaces. The techniques involved have
many other applications beyond the one discussed here;
in particular, they allow a ``direct
approach'' to regularity theory for minimizers of coercive
functionals, without having to consider the Euler equation.
We begin with a simple result.
\begin{lemma}\label{lem:1}
For any $u\in H^1(B_R(x_0))$, and any $r<1$, $\int_{B_r}|u-c|^2dx$
is minimum when $c=u_{x_0,r}$ and Poincar\'e's inequality
\[ \int_{B_r(x_0)}|u-u_{x_0,r}|^2dx \leq C\int_{B_r(x_0)}|\nabla
u|^2dx
\]
holds.
If we assume in addition that $u$ is harmonic, and $0<a<1$, we have
\[ \int_{B_{ar}(x_0)}|\nabla u|^2\leq
c(a)r^{-2}\int_{ar<|x-x_0|<r}|u-u_{x_0,r}|^2dx.
\]
The r.h.s.\ is in particular estimated by
$c(a)r^{-2}\int_{B_{r}(x_0)}|u|^2$.
\end{lemma}
\begin{proof}
Poincar\'e's inequality is classical, see \emph{e.g.} \cite{B}.
Let $x_0=0$ and choose a smooth, nonnegative function $\eta$
supported by $B_r$, equal to $1$ for $|x|\leq ar$, such that
$|\nabla\eta|\leq C/r$.\footnote{It suffices to find such a
function $\eta_0$ for the case $r=1$ and then let
$\eta(x)=\eta_0(x/r)$.} Multiply the Laplace equation by
$(u-m)\eta^2$, where $m$ is any constant. We find, integrating by
parts,
\[ \int_{B_r}[\eta^2|\nabla u|^2+2(u-m)\nabla\eta\cdot\nabla u]dx=0,
\]
hence, estimating $\nabla\eta$ by $C/r$ and using H\"older's
inequality, we find
\[ \int_{\eta =1}|\nabla u|^2dx\leq \int\eta^2|\nabla u|^2dx
\leq \frac C{r^2}\int_{\nabla\eta\neq 0}|u-m|^2dx.
\]
Taking $m$ to be the average of $u$ on $B_r$, we obtain in
particular the desired inequality. The inequality corresponding to
$m=0$ is also occasionally useful.
\end{proof}
\vskip 1em
Since the derivatives of a harmonic function are themselves
harmonic, this result implies that higher order derivatives are
locally square-integrable; the Sobolev inequality then shows
easily that any harmonic function is smooth. We now turn to a more
precise estimate which enables one to compare a harmonic function and
its spherical mean.
\btl{cm1} Let $u$ be harmonic in the ball of radius $R_0$ about
$x_0\in\RR^n$. If $0<r<R<R_0$, we have
\begin{align}
\sup_{B_{R/2}(x_0)}|u|^2 &\leq CR^{-n}\int_{B_R(x_0)}|u|^2dx \\
\int_{B_r(x_0)}|u|^2dx &\leq C(\frac rR)^n\int_{B_R(x_0)}|u|^2dx
\\ \int_{B_r(x_0)}|u-u_{x_0,r}|^2dx &\leq C(\frac
rR)^{n+2}\int_{B_R(x_0)}|u-u_{x_0,R}|^2dx
\end{align}
where $C$ is independent of $u$, $r$, $R$ and $R_0$. \etl
\begin{proof}
We may take $x_0=0$. It suffices to prove the first inequality for
$R=1$ and scale variables. If $k$ is an integer, we
find, applying the preceding lemma repeatedly, we find
\[ \int_{B_{1/2}}|\nabla^k u|^2dx\leq c(k)\int_{B_1}u^2 dx.
\]
The result follows by the Sobolev inequality.

Regarding the last two inequalities, it suffices to prove them for
$r\leq R/2$ since they are obvious for $r\geq R$. In that case, we
have, using the first inequality,
\[\int_{B_r}|u|^2dx \leq
\omega_nr^n\sup_{B_{R/2}}|u|^2\leq C(\frac rR)^n\int_{B_R}|u|^2dx
\]
as desired. Similarly, since the derivatives of $u$ are also
harmonic, Poincar\'e's inequality yields
\[ \int_{B_r}|u-u_{x_0,r}|^2dx \leq C\rho^2\int_{B_r}|\nabla u|^2dx
\leq C\rho^2(\frac rR)^n\int_{B_{3R/4}(x_0)}|\nabla u|^2dx.
\]
We conclude using lemma \ref{lem:1}.
\end{proof}
\vskip 1em
We turn to the estimation of second derivatives of the solutions
of Poisson's equation $-\Delta v=f$. It is equivalent to seek an estimate
for the \emph{first} derivatives of solutions of $-\Delta u=\pa_k f$.
It turns out to be convenient
to consider more generally the problem
\begin{equation}\label{eq:14}
\Delta u+\sum_k\pa_k f^k =0,
\end{equation}
where $u\in H^1(B_R)$ and the $f^k$ are of class $C^\alpha$.
Recall that function of class $C^\alpha$
correspond to the class $\mathcal L^{2,\lambda}$ for
$\lambda=n+2\alpha$. This suggests the following theorem.
\btl{cm2} Assume that $\mathbf f:=(f_1,\dots,f_n)\in\mathcal L^{2,\lambda}$
with $0\leq\lambda<n+2$, and that $u\in
 H^1(B_R)$ solves (\ref{eq:14}),
then $\nabla u$ is locally of class
$\mathcal L^{2,\lambda}$. In particular, if the $f^k$ are locally
$C^\alpha$, with $0<\alpha<1$, so is $\nabla u$. \etl
\begin{proof}
We must analyze the behavior of the integrals
\[ F(r):=\int_{B_r(x_0)}|\nabla u-(\nabla u)_{x_0,r}|^2dx,
\]
defined for given $x_0\in B_R$ and $r<R-|x_0|$, as $r\to 0$. We
first prove the estimate
\begin{equation}\label{eq:cm}
   F(\rho)\leq A(\rho/r)^{n+2}F(r)+Br^\lambda.
\end{equation}
for $\rho<r$. We then deduce from it an estimate of the form
$F(r)=O(r^\lambda)$, from which the result follows.

Consider the solution of $\Delta v=0$ in $B_r(x_0)$ such that
$u-v\in H^1_0(B_r(x_0))$. From theorem \ref{th:cm1}, we have the
inequality
\[ \int_{B_\rho}|\nabla v-(\nabla v)_{x_0,\rho}|^2dx\leq
C(\frac \rho r)^{n+2}\int_{B_r}|\nabla v-(\nabla
v)_{x_0,\rho}|^2dx.
\]
For any $w\in H^1_0(B_r(x_0))$, we have, writing $\pa_k
f^k=\pa_k(f^k-f^k_{x_0,r})$,
\[\int_{B_r}\nabla (u-v)\cdot\nabla w\;dx=
-\int (\mathbf f-\mathbf f_{x_0,r})\cdot\nabla w\; dx.
\]
Taking $w=u-v$, we find
\[\int_{B_r}|\nabla (u-v)|^2\leq \int_{B_r}|\mathbf f-\mathbf
f_{x_0,r}|^2dx.
\]
From now on, we omit the mention of the point $x_0$ in averages.
Since $\nabla u-(\nabla
u)_\rho=[\nabla v-(\nabla v)_\rho]+[\nabla w-(\nabla w)_\rho]$,
and
\[ \int_{B_\rho}|\nabla w-(\nabla w)_\rho|^2dx\leq
\int_{B_\rho}|\nabla w|^2dx\leq \int_{B_r}|\nabla w|^2dx,
\]
we find
\begin{align*}
\int_{B_\rho}|\nabla u-(\nabla u)_\rho|^2dx &\leq
2\int_{B_\rho}|\nabla v-(\nabla v)_\rho|^2dx +2\int_{B_r}|\nabla
w|^2dx     \\ &\leq C(\frac \rho r)^{n+2}\int_{B_r}|\nabla
v-(\nabla v)_r|^2dx +C\int_{B_r}|\mathbf f-\mathbf f_r|^2dx.
\end{align*}
The second term is $O(r^\lambda)$ thanks to the hypothesis on
$\mathbf f$. We now estimate the first term in terms
of $F(r)$. To this end, we need the following result:
\begin{lemma}
$\int_{B_r}|\nabla v-(\nabla v)_r|^2dx\leq \int_{B_r}|\nabla
u-(\nabla u)_r|^2dx$.
\end{lemma}
\begin{proof}
Since $u-v$ is in $H^1_0$, we have
\[ \int_{B_r}|\nabla v|^2dx\leq \int_{B_r}|\nabla u|^2dx,
\]
and, as soon as $\mathbf g$ is constant
\[ \int_{B_r}\nabla (v-u)\cdot \mathbf g dx=0.
\]
It follows that
\begin{align*}
\int_{B_r}|\nabla v&-(\nabla v)_r|^2dx - \int_{B_r}|\nabla
u-(\nabla u)_r|^2dx  \\ &=\int_{B_r}\nabla
(v-u)\cdot(\nabla(v+u)-(\nabla(v+u))_r)dx \\ &= \int_{B_r}(|\nabla
v|^2-|\nabla u|^2)dx\leq 0,
\end{align*}
QED.
\end{proof}
\vskip 1em
The proof of inequality (\ref{eq:cm})is now complete. To
conclude the proof of the theorem, we argue as follows: Fix
$\gamma\in(\lambda,n+2)$, and choose $t\in(0,1)$ such that
\[ 2A t^{n+2}\leq t^\gamma.
\]
Fix $j$ such that $t^{j+1}r<\rho\leq t^jr$. We find, since $F$ is
non-decreasing,
\begin{align*}
F(\rho)\leq F(t^j) &\leq t^\gamma F(t^{j-1}r)+B(t^{j-1}r)^\lambda
\\
                   &\leq t^\gamma[t^\gamma
                   F(t^{j-2}r)+B(t^{j-2}r)^\lambda]+B(t^{j-1}r)^\lambda\\
&=  t^{2\gamma}F(t^{j-2}r)+Br^\lambda
t^{(j-1)\lambda}[1+t^{\gamma-\lambda}]\\ &\leq\cdots\leq
t^{j\gamma}F(r)+Br^\lambda
\frac{t^{(j-1)\lambda}}{1-t^{\gamma-\lambda}}\\ &\leq
t^{-\gamma}(\frac \rho r)^\gamma F(r)
      + \frac {Bt^{-2\lambda}}{1-t^{\gamma-\lambda}}(\frac \rho
      r)^\lambda.
\end{align*}
Since $\gamma>\lambda$, this implies $F(\rho)=O(\rho^\lambda)$ as
desired. This gives the H\"older regularity of the gradient of
$u$.
\end{proof}
\begin{remark}
For more general problems, it is useful to note that the last part
of the argument also applies in the more general situation in
which $F$ is a non-negative, non-decreasing function satisfying
\[ F(\rho)\leq A[(\rho/r)^a+\ep]F(r)+Br^b,
\]
for $0<\rho<r\leq R$, with $a>b$. Taking $t$ as before, we find
that if $\ep\leq t^a$, then $F$ satisfies an inequality of the
form
\[ F(\rho)\leq c(a,b,A)[(\rho/r)^bF(R)+B\rho^b].
\]
\end{remark}
\begin{remark}
A similar argument may be applies to non-divergence operators with
H\"older-continuous coefficients, using the previous remark. As
expected, the argument consists in writing the operator as the sum
of a constant-coefficient operator, and an operator with small
coefficients.
\end{remark}

\subsection{Other methods}

We briefly outline two other approaches.

\subsubsection{A regularization method}

For any standard mollifier $\rho_\ep(x)=\ep^{-n}\rho(x/\ep)$,
consider, for any $u(x)$,
\[ u_\ep(x)=\rho_\ep * u=\int \rho(z)u(x-\ep z)dz.
\]
It is easy to see that $u(x)-u_\ep(x)=O(\ep^\alpha)$ if $u\in
C^\alpha(\RR^n)$, and that, similarly, the derivatives of $u_\ep$
with respect to $\ep$ or $x$ are $O(\ep^{\alpha-1})$. Conversely,
if these derivatives are $\leq M\ep^{\alpha-1}$, it turns out that
one may estimate the H\"older constant of $u$: first of all,
\[  |u(x)-u(x_\ep)|
\leq \ep\int_0^1 |\frac{\pa u}{\pa\ep}|(x,\ep\sigma)d\sigma \leq
M\ep^\alpha\int_0^1 \frac{d\sigma}{\sigma^{1-\alpha}}=
\frac{M\ep^\alpha}{\alpha}.
\]
We now find the estimate
\begin{align*}
|u(x)-u(y)| &= |u(x)-u(x_\ep)|+|u(x_\ep)-u(y_\ep)|+|u(y_\ep)-u(y)|
\\
            &= \frac{2}{\alpha}M\ep^\alpha+
               |x-y||\nabla_x u_\ep(z)|
\end{align*}
for some $z\in[x,y]$. One then estimates $|\nabla_x u_\ep(z)|$ by
$M\ep^{\alpha-1}$ and takes $\ep=|x-y|$. One also proves that
mixed derivatives of $u_\ep$ with respect to $x$ and $\ep$ are
controlled by the H\"older norm of second derivatives of $u$ with
respect to $x$.

One then considers equation $-\Delta u=f(x)=f(x_0)+g(x)$, where
$|g|\leq R^\alpha[f]_{\alpha,B_R(x_0)}$. We have $-\Delta
u_\ep=g_\ep$. If $\nabla^2$ represents any second-order
derivative, in $x$ and $\ep$, we find $-\Delta
\nabla^2u_\ep=\nabla^2g_\ep$. Applying the interior $C^1$ estimate
theorem \ref{th:c1-mp}, and letting $R=N\ep$, where $N$ is to be
chosen later, one gets
\[
\ep^{1-\alpha}|\nabla\pa_{ij}u(x_0,\ep)|\leq
C\{N^{\alpha-1}[\nabla^2_xu]_\alpha
+N^{\alpha+1}R^{-\alpha}\sup_{B_{(1+N)\ep}(x_0)}|g|\}
\]
where one has estimated $\sup_{B_R(x_0)}|\nabla^2 g|$ by
$\ep^{-2}\sup_{B_{R+\ep}(x_0)}|g|$. This quantity is itself
estimated by $[f]_\alpha(R+\ep)^\alpha$. Taking $N$ so large that
$CN^{\alpha-1}<1/2$, we find an estimate of $[\nabla^2_x
u]_\alpha$, as desired.

\subsubsection{Blow-up method}

We sketch the idea of the proof of the interior estimate for the
Laplacian; a similar idea, with somewhat more complicated details,
applies to other situations.

Assume there is no estimate of the form $[\nabla^2u]_\alpha\leq
C[\Delta u]_\alpha$ for functions of class $C^{2+\alpha}(\RR^n)$.
In that case, there must be some sequence $u_k$ such that
\[ [\nabla^2u_k]_\alpha=1 > 2k [\Delta u_k]_\alpha.
\]
We may therefore find indices $i$ and $j$, and sequences $x_k$,
$a_k$ of vectors in $\RR^n$ such that
\[
1\geq \frac 1{|a_k|^\alpha}|\pa_{ij}u_k(x_k)-u_k(x_k+a_k)| \geq
\frac 12 \geq k[\Delta u_k]_\alpha.
\]
Subtracting an affine function, we may assume that $u_k$ and its
first-order derivatives vanish at the point $x_k$. Subtracting a
quadratic function, we may also assume that $\Delta u_k$ vanishes
at $x_k$. Performing a ($k$-dependent) rotation of axes, we may
also assume that $a_k=(h_k,0,\dots,0)$. Considering
$v_k(y)=h_k^{-(2+\alpha)}u_k(x_k+h_k y)$, we see that $[\Delta
v_k]_\alpha=[\Delta u_k]_\alpha\to 0$, while $v_k$ and its
first-order derivatives vanish at the origin and grows at most
like $|y|^{2+\alpha}$ at infinity. In addition, we have
\begin{equation}\label{eq:spr}
|\pa_{ij}v_k(0)-v_k(e_1)|\geq \frac12 .
\end{equation}
where $e_1=(1,0,\dots,0)$. After extraction of a subsequence, we
are left with a harmonic function which grows at most like
$|y|^{2+\alpha}$, and satisfies equation (\ref{eq:spr}). A variant of
the Liouville property ensures that $v$ is quadratic, which
contradicts (\ref{eq:spr}).

\section{Perturbation of coefficients}

\subsection{Basic estimate}

Working on a relatively compact subset $\Omega'$ of $\Omega$,
we may assume that $[u]^*_{2+\alpha}<\infty$;
since the constants in the various
inequalities will not depend on the choice of $\Omega$, the full
result will follow.

Consider $x_0\in\Omega$ and let $r=\theta d(x_0)$ with $\theta\leq
1/2$. Let $L_0=\sum_{ij}a^{ij}(x_0)\pa_{ij}$ (the ``tangential
operator,'' with coefficients `` frozen'' at $x_0$). We define
\[ F:= L_0u=
\sum_{ij}(a^{ij}(x_0)-a^{ij}(x))\pa_{ij}u-\sum b^i\pa_i u-cu+f.
\]
We apply the constant-coefficient interior estimates on the ball
$B_r(x_0)$. Let $y_0\neq x_0$ such that $d(y_0)\geq d(x_0)$.

If $|x_0-y_0|<r/2$, we have
\[ (\frac r2)^{2+\alpha}[\nabla^2u]_{\alpha,x_0,y_0}\leq
   C(\sup|u|+\sup_{B_r}|r^2F|+\sup_{B_r\times B_r} r^{2+\alpha}
   \frac{|F(x)-F(y)|}{|x-y|^\alpha}).
\]
Therefore,
\begin{equation}\label{eq:theta-est}
d(x_0)^{2+\alpha}[\nabla^2u]_{\alpha,x_0,y_0}\leq
C\theta^{-2-\alpha}(\sup |u|+\|F\|^{(2)}_{\alpha,B_r}).
\end{equation}
If $|x_0-y_0|\geq r/2$, we have
\begin{equation}\label{eq:theta-est2}
 d(x_0)^{2+\alpha}[\nabla^2u]_{\alpha,x_0,y_0}\leq
   2[u]^*_2\frac{d(x_0)^\alpha}{|x_0-y_0|^\alpha}\leq
   2[u]^*_2(\frac 2\theta)^\alpha.
\end{equation}
The issue is therefore the estimation of
$\|F\|^{(2)}_{\alpha,B_r}$ in terms of norms of $u$ and its
derivatives over $\Omega$.

For clarity, we begin with three lemmas.
\begin{lemma}
$\|uv\|^{(s+t)}_{\alpha,\Omega}\leq
\|u\|^{(s)}_{\alpha,\Omega}\|v\|^{(t)}_{\alpha,\Omega}$.
\end{lemma}
\begin{proof}
Direct verification.
\end{proof}
\begin{lemma}
If $r=\theta d(x,\pa \Omega)$, with $0<\theta\leq 1/2$ (so that
$B_r(x)\subset \Omega$), we have
\begin{align}
\|\nabla^2u\|^{(2)}_{\alpha,B_r} &\leq
8[\theta^2\|\nabla^2u\|^{*}_{2,\Omega}+\theta^{2+\alpha}[u]^*_{2+\alpha,\Omega}]\\
\|f\|^{(2)}_{\alpha,B_r} &\leq
8\theta^2\|f\|^{(2)}_{\alpha,\Omega}
\end{align}
\end{lemma}
\begin{proof}
We need to estimate, for $y\in B_r(x)$, $d(y,\pa B_r(x))$ and
$d_{x,y,B_r}$ in terms of the corresponding distances relative to
$\Omega$. On the one hand, $d(y,\pa B_r)\leq r-|x-y|\leq r=\theta
d(x)$. On the other hand, if $z\in B_r(x)$ and $d(y,\pa
B_r(x))\leq d(z,\pa B_r(x))$, we have $d_{y,z,B_r}\leq \theta
d(x)$ and also $d(y)\geq d(y,\pa B_r(x))\geq (1-\theta)d(x)$; it
follows that $d(x)\leq (1-\theta)^{-1}d_{x,y,\Omega}$. Therefore,
\[ d(y,\pa B_r) \leq \theta d(x)
\]
and
\[ d_{y,z,B_r}\leq \frac\theta{1-\theta}d_{y,z,\Omega}.
\]
The two desired inequalities follow.
\end{proof}
\begin{lemma}
If $x\in B_r(x_0)$ with $r=\theta d(x_0)$, with $0<\theta\leq 1/2$,
we have
\[ \|a(x)-a(x_0)\|^{(0)}_{\alpha,B_r}\leq
C\theta^\alpha[a]^*_{\alpha,\Omega}.
\]
\end{lemma}
\begin{proof}
If $d(x)\leq d(y)$ and $|x-y|\leq r =\theta d(x_0)$ with
$\theta\leq 1$,
\[ |a(x)-a(y)|\leq d(x)^\alpha
\frac{|a(x)-a(y)|}{|x-y|^\alpha}(\frac{|x-y|}{d(x)})^\alpha\leq
C\theta^\alpha [a]^{(0)}_\alpha
\]
since $(1-\theta)d(x_0)\leq d(x)\leq (1+\theta)d(x_0)$. Therefore,
estimating $|a(x)-a(x_0)|$ by $r^\alpha [a]^*_{\alpha,\Omega}$, we
find the announced inequality.
\end{proof}
\vskip 1em
We now resume the proof of the estimate of $[\nabla^2u]_{\alpha}$:
first,
\begin{align*}
\|(a(x)-a(x_0))\nabla^2u(x)\|^{(2)}_{\alpha,B_r}&\leq
\|a(x)-a(x_0)\|^{(0)}_{\alpha,B_r}
\|\nabla^2u\|^{(2)}_{\alpha,B_r}\\ &\leq
C\theta^{2+\alpha}\|a\|^{(0)}_{\alpha,\Omega}
(\|\nabla^2u(x)\|^{*}_{\alpha,\Omega}+\theta^\alpha[u]^*_{2+\alpha,\Omega}).
\end{align*}
Similarly,
\begin{align*}
\|b\nabla u(x)\|^{(2)}_{\alpha,B_r}
  &\leq 8\theta^2\|b\nabla u\|^{(2)}_{\alpha,\Omega} \\
  &\leq 8\theta^2\|b\|^{(1)}_{\alpha,\Omega}\|\nabla  u\|^{(1)}_{\alpha,\Omega}\\
  &\leq C\theta^2\|b\|^{(1)}_{\alpha,\Omega}\{\theta^{2\alpha}[u]^*_{2+\alpha,\Omega}+\sup |u|\}.
\end{align*}
Finally,
\begin{align*}
\|cu\|^{(2)}_{\alpha,B_r}
  &\leq 8\theta^2\|cu\|^{(2)}_{\alpha,\Omega}\leq 8\theta^2
  \|c\|^{(2)}_{\alpha,\Omega}\|u\|^{(0)}_{\alpha,\Omega}\\
  &\leq 8\theta^2\{\theta^{2\alpha}[u]^*_{2+\alpha,\Omega}+\sup |u|\}.
\end{align*}
It follows that
\[ \|F\|^{(2)}_{\alpha,B_r}\leq
C\theta^{2+2\alpha}[u]^*_{2+\alpha,\Omega}+c(\theta)(\sup|u|+\|f\|^{(2)}_{\alpha,\Omega}).
\]
Therefore, using this inequality in (\ref{eq:theta-est})
and (\ref{eq:theta-est2}), we find
\[ d(x_0)^{2+\alpha}[u]^*_{2+\alpha,\Omega}\leq C\theta^\alpha
  [u]^*_{2+\alpha,\Omega}+c'(\theta)(\sup|u|+\|f\|^{(2)}_{\alpha,\Omega}).
\]
The desired estimate on $[u]^*_{2+\alpha,\Omega}$ follows.

\subsection{Estimates up to the boundary}

The potential-theoretic argument extends easily to the case of
Poisson's equation on the half-ball for the following reason: if
we apply the formula for the second-order derivatives of the
Newtonian potential (theorem \ref{th:wij}) to the case in which
$\Omega$ is the half-ball $B_R\cap\{x_n>0\}$, we find that the
contribution to the boundary integral of the part of the boundary
on which $x_n=0$ vanishes if $j<n$, because the component $n_j$ of
the outward normal then vanishes. The subsequent argument
therefore goes through without change, and yields the H\"older
continuity up to the boundary of all second-order derivatives of
$u$ except $\pa^2_{x_n}u$; but the latter is given in terms of the
former using Poisson's equation. We therefore obtain the
$C^{2+\alpha}$ estimates up to the boundary for the Newtonian
potential of a density $f$ of class $C^{\alpha}$ in the half-ball.

To obtain regularity up to $x_n=0$ for the solution of the
Dirichlet problem on the half-ball, we use \emph{Schwarz'
reflection principle}
\begin{lemma}
Let $f$ be of class $C^\alpha$ in the closed half-ball. If $u$ is
of class $C^2$ on the open half-ball of radius $R$, is continuous
on the closed ball, satisfies $\Delta u=f$ in the half-ball, and
vanishes for $x_n=0$, it may be extended to the entire ball as a
solution of an equation of the form $\Delta u=f_1$. In particular,
$u$ is of class $C^{2+\alpha}$ on any compact subset of the closed
half-ball which does not meet the spherical part of its boundary.
\end{lemma}
\begin{proof}
Write $x=(x',x_n)$, and extend $f$ to an even function $f_1$ on
the ball. Using the inequality $a^\alpha+b^\alpha\leq
2(a+b)^\alpha$, we see that $f_1$ is of class $C^\alpha$. Now, the
Newtonian potential of $f_1$ does not satisfy the Dirichlet
boundary condition. We therefore consider
\[ W(x):=\int_{B_R\cap\{x_n>0\}}
[g(x-y)-g(x-\tilde y)]f(y)dy,
\]
where $\tilde y=(y',-y_n)$ is the reflection of $y$
across $\{x_n=0\}$.
It is easy to see that $\Delta W=0$ in the half-ball, and that
$W=0$ for $x_n=0$. It is also of class $C^{2+\alpha}$ by the
variant of theorem \ref{th:wij} already indicated. We now consider
$V:=u-W$, which is harmonic in the half-ball, and vanishes for
$x_n=0$. Extend $V$ to an \emph{odd} function of $x_n$ on the
entire ball. Consider the solution of the Dirichlet problem on the
ball with boundary data equal to $V$. This problem has a unique
solution $V^*$ by the Poincar\'e-Perron method---which is
independent of Schauder theory. Since $-V^*(x',-x_n)$ solves the
same problem, we find that $V^*$ must be odd with respect to
$x_n$. Therefore $V^*$ is also the solution of the Dirichlet
problem on the half-ball, with boundary value given by $V$ on the
spherical part of the boundary, and value zero on the flat part of
the boundary (where $x_n=0$). Therefore, $V^*$ must be equal to
$V$ on the half-ball, and therefore on the ball as well. This
proves that $V=V^*$ has the required regularity up to $x_n=0$, as
desired.
\end{proof}

The perturbation from constant to variable coefficients then
proceeds by a variant of the argument used for the interior
estimates \cite{GT,DN,ADN}.

\section{Fuchsian operators on $C^{2+\alpha}$
domains}

We now consider operators satisfying an asymptotic scale invariance
condition near the boundary. These operators arise naturally as
local models near singularities through the process of Fuchsian
Reduction \cite{SK-fr}. We develop the basic estimates for such
operators without condition on the sign of the lower-order terms.
A typical example of the more precise theorems one obtains under
such conditions is given in theorem \ref{th:III}. We distinguish
two types of Fuchsian operators.

An operator $A$ is said to be \emph{of type (I)} (on a given
domain $\Omega$) if it can be written
\[ A=\pa_i(d^2a^{ij}\pa_{j})+db^i\pa_i+c, \]
with $(a^{ij})$ uniformly elliptic and of class $C^\alpha$, and
$b^i$, $c$ bounded.
\begin{remark}
One can also allow terms of the type $\pa_i(b^{\prime i}u)$ in
$Au$, if $b^{\prime i}$ is of class $C^\alpha$, but this
refinement will not be needed here.
\end{remark}

An operator is said to be \emph{of type (II)} if it can be written
\[ A= d^2a^{ij}\pa_{ij}+db^i\pa_i+c, \]
with $(a^{ij})$ uniformly elliptic and $a^{ij}$, $b^i$, $c$ of
class $C^\alpha$.
\begin{remark}
One checks directly that types (I) and (II) are invariant under
changes of coordinates of class $C^{2+\alpha}$. In particular, to
check that an operator is of type (I) or (II), we may work
indifferently in coordinates $x$ or $(T,Y)$ defined in section
\ref{sec:prel}. All proofs will be performed in the $(T,Y)$
coordinates; an operator is of type (II) precisely if it has the
above form with $d$ replaced by $T$, and the coefficients
$a^{ij}$, $b^i$, $c$ are of class $C^\alpha$ as functions of $T$
and $Y$; a similar statement holds for type (I).
\end{remark}

The basic results for type (I) operators are
\btl{FIa} If $Ag=f$, where $f$ et $g$ are
bounded and $A$ is of type (I) on $\Omega'$, then $d\nabla g$ is
bounded, and $dg$ and $d^2\nabla g$ belong to
$C^\alpha(\Omega'\cup\pa\Omega)$.
\etl
\btl{FIb} If $Ag=df$, where
$f$ and $g$ are bounded, $g=O(d^\alpha)$, and $A$ is of type (I)
on $\Omega'$, then $g\in C^{\alpha}(\Omega'\cup\pa\Omega)$ and
$dg\in C^{1+\alpha}(\Omega'\cup\pa\Omega)$
\etl
These two results are proved in the next subsection. The main
result for type (II) operators is:
\btl{FIIa} If $Ag=df$, where
$f\in C^\alpha(\Omega'\cup\pa\Omega)$, $g=O(d^\alpha)$, and $A$ is
of type (II) on $\Omega'$, then $d^2 g$ belongs to
$C^{2+\alpha}(\Omega'\cup\pa\Omega)$. \etl
\begin{proof}
The assumptions ensure that $a^{ij}\pa_{ij}(d^2f)$ is
H\"older-continuous and that $f$ is bounded; $d^2f$ therefore
solves a Dirichlet problem to which the Schauder estimates apply
near $\pa\Omega$. Therefore $d^2f$ is of class $C^{2+\alpha}$ up
to the boundary. Since we already know that $f\in
C^\alpha(\overline\Omega_\delta)$ and $df$ is of class
$C^{1+\alpha}(\overline\Omega_\delta)$, we have indeed $f$ of
class $C_\sharp^{2+\alpha}(\overline\Omega_{\delta'})$ for
$\delta'<\delta$.
\end{proof}

Let $\rho>0$ and $t\leq 1/2$. Throughout the proofs, we shall use
the sets
\begin{align*}
Q   &=\{ (T,Y) : 0\leq T\leq 2 \text{ and } |y|\leq 3\rho\},\\ Q_1
&=\{ (T,Y) : \frac14\leq T\leq2 \text{ and } |y|\leq 2\rho\},\\
Q_2 &=\{ (T,Y) : \frac12\leq T\leq 1 \text{ and }
|y|\leq\rho/2\},\\ Q_3 &=\{ (T,Y) : 0\leq T\leq \frac12 \text{ and
} |y|\leq\rho/2\}.
\end{align*}
We may assume, by scaling coordinates, that $Q\subset\Omega'$. It
suffices to prove the announced regularity on $Q_3$.


\subsection{First ``type (I)'' result}
\label{sec:e7}

We prove theorem \ref{th:FIa}.

Let $Af=g$, with $A$, $f$, $g$ satisfying the assumptions of the
theorem over $Q$, and let $y_0$ be such that $|y_0|\leq \rho$.

For $0<\eps\leq 1$, and $(T,Y)\in Q_1$, let
\[ f_\eps(T,Y)=f(\eps T,y_0+\eps Y),
\]
and similarly for $g$ and other functions. We have
$f_\eps=(Ag)_\eps=A_\eps f_\eps$, where
\[
A_\eps=\pa_i(T^2a^{ij}_\eps\pa_{j})+Tb^i_\eps\pa_i+c_\eps
\]
is also of type (I), with coefficient norms independent of $\eps$
and $y_0$, and is uniformly elliptic in $Q_1$.

Interior estimates give
\begin{equation}\label{eq:c-1-alpha}
\|g_\eps\|_{C^{1+\alpha}(Q_2)}\leq
M_1:=C_1(\|f_\eps\|_{L^{\infty}(Q_1)}+
\|g_\eps\|_{L^{\infty}(Q_1)}).
\end{equation}
The assumptions of the theorem imply that $M_1$ is independent of
$\eps$ and $y_0$.

We therefore find,
\begin{align}
|\eps\nabla g(\eps T, y_0+\eps Y)| &\leq M_1,\\ \eps|\nabla g(\eps
T, y_0+\eps Y)-\nabla g(\eps T', y_0)| &\leq M_1
   (|T-T'|+|Y|)^\alpha\,
\end{align}
if $\frac12\leq T, T'\leq 1$ and $|Y|\leq\rho/2$. It follows in
particular, taking $Y=0$, $\eps=t\leq 1$, $T=1$, and recalling
that $|y_0|\leq\rho$, that
\begin{equation}
  |t\nabla g(t,y)|\leq M_1\text{ if } |y|\leq \rho, t\leq 1.
\end{equation}
This proves the first statement in the theorem.

Taking $\eps=2t\leq 1$, $T=1/2$, and letting $y=y_0+\eps Y$,
$t'=\eps T'$,
\[
  2t|\nabla g(t,y)-\nabla g(t',y_0)|\leq
      M_1(|t-t'|+|y-y_0|)^\alpha (2t)^{-\alpha}
\]
for $|y-y_0|\leq \rho t$ and $t\leq t'\leq 2t\leq 1$.

Let us prove that
\begin{equation}\label{eq:est}
 |t^2\nabla g(t,y)-t^{\prime 2}\nabla g(t',y_0)|\leq
      M_2(|t-t'|+|y-y_0|)^\alpha
\end{equation}
for $|y|, |y_0|\leq \rho$, and $0\leq t\leq t'\leq \frac12$, which
will prove
\[  t^2\nabla g\in C^\alpha(Q_3).
\]
It suffices to prove this estimate in the two cases: (i) $t=t'$
and (ii) $y=y_0$; the result then follows from the triangle
inequality. We distinguish three cases.
\begin{enumerate}
  \item If $t=t'$, we need only consider the case $|y-y_0|\geq\rho
  t$. We then find
  \[ t^2|\nabla g(t,y)-\nabla g(t,y_0)|\leq
      2M_1t\leq 2M_1|y-y_0|/\rho.
  \]
  \item If $y=y_0$ and $t\leq t'\leq 2t\leq 1$, we have $t+t'\leq
  2t'$, hence
  \begin{align*}
      |t^2\nabla g(t,y_0)&-t^{\prime 2}\nabla g(t',y_0)|\\
         &\leq t^2|\nabla g(t,y_0)-\nabla g(t',y_0)|
            +|t-t'|(t+t')|\nabla g(t',y_0)|    \\
         &\leq M_12^{-1-\alpha}t^{1-\alpha}|t-t'|^\alpha
                +2M_1|t-t'|                    \\
         &\leq M_2|t-t'|^\alpha.
  \end{align*}
  \item If $y=y_0$, and $2t\leq t'\leq 1/2$, we have $t+t'\leq
  3(t'-t)$, and
  \begin{align*}
      |t^2\nabla g(t,y_0)-t^{\prime 2}\nabla g(t',y_0)|
         &\leq  M_1(t+t') \\
         &\leq 3M_1 |t-t'|.
  \end{align*}
\end{enumerate}
This proves estimate (\ref{eq:est}).

On the other hand, since $g$ and $T\nabla g$ are bounded over
$Q_3$,
\[  Tg\in \mathrm{Lip}(Q_3)\subset C^\alpha(Q_3).
\]
This completes the proof of theorem \ref{th:FIa}.


\subsection{Second ``type (I)'' result}
\label{sec:e8}

We prove theorem \ref{th:FIb}.

The argument is similar, except that $M_1$ is now replaced by
$M_3\eps^\alpha$, with $M_3$ independent of $\eps$ and $y_0$. It
follows that
\begin{equation}
  |t\nabla g(t,y)|\leq M_3t^\alpha\text{ if } |y|\leq \rho, t\leq 1.
\end{equation}
Taking $\eps=2t\leq 1$, $T=1/2$, and letting $y=y_0+\eps Y$,
$t'=\eps T'$, and noting that
$\eps^\alpha(|T-T'|+|Y|)^\alpha=(|t-t'|+|y-y_0|)^\alpha$, we find
\[
  2t|\nabla g(t,y)-\nabla g(t',y_0)|\leq
      M_3 (|t-t'|+|y-y_0|)^\alpha
\]
for $|y-y_0|\leq \rho t$ and $t\leq t'\leq 2t\leq 1$. Let us prove
that
\begin{equation}\label{eq:est2}
 |t\nabla g(t,y)-t'\nabla g(t',y_0)|\leq
      M_4(|t-t'|+|y-y_0|)^\alpha
\end{equation}
for $|y|, |y_0|\leq \rho$, and $0\leq t\leq t'\leq \frac12$, which
will prove
\[  T\nabla g\in C^\alpha(Q_3).
\]
We again distinguish three cases.
\begin{enumerate}
  \item If $t=t'$, $|y-y_0|\geq\rho t$, we find
  \[ t|\nabla g(t,y)-\nabla g(t,y_0)|\leq
      2M_3t^\alpha\leq 2M_3(|y-y_0|/\rho)^\alpha.
  \]
  \item If $y=y_0$ and $t\leq t'\leq 2t\leq 1$, we have $|t-t'|\leq
  t\leq t'$, hence
  \begin{align*}
      |t\nabla g(t,y_0)-t'\nabla g(t',y_0)|
         &\leq \frac12 M_3|t-t'|^\alpha
                +|t-t'||\nabla g(t',y_0)|             \\
         &\leq M_3|t-t'|^\alpha(\frac12+t'^{1-\alpha}t'^{\alpha-1})
               \leq 2M_3|t-t'|^\alpha.
  \end{align*}
  \item If $y=y_0$, and $2t\leq t'\leq 1/2$, we have $t\leq t'\leq
  3(t'-t)$, and
  \begin{align*}
      |t\nabla g(t,y_0)-t'\nabla g(t',y_0)|
         &\leq  M_3(t^\alpha+t^{\prime \alpha}) \\
         &\leq 2M_3 (3|t-t'|)^\alpha.
  \end{align*}
\end{enumerate}
Estimate (\ref{eq:est2}) therefore holds.

The same type of argument shows that
\[  g\in C^\alpha(Q_3). \]
In fact, we have, with again $\eps=2t$,
$\|g_\eps\|_{C^\alpha(Q_2)}\leq M_5\eps^\alpha$, where $M_5$
depends on the r.h.s.\ and the uniform bound assumed on $f$. This
implies
\[ |g(t,y)-g(t',y_0)|\leq M_5 (|t-t'|+|y-y_0|)^\alpha,
\]
if  $t\leq t' \leq 2t\leq 1\text{ and } |y-y_0|\leq\rho t$. The
assumptions of the theorem yield in particular
\[ |g(t,y)|\leq M_5 t^\alpha,
\]
for $t\leq 1/2$ and $|y|\leq \rho$.

If $\rho t\leq |y-y_0|\leq \rho$, and $t\leq 1/2$, we have
\[ |g(t,y)-g(t,y_0)|\leq 2M_5t^\alpha
     \leq 2M_5\left(\frac{|y-y_0|}{\rho}\right)^\alpha.
\]

If $2t\leq t'\leq 1/2$ and $y=y_0$,
\[ |g(t,y_0)-g(t',y_0)|
  \leq M_5(t^\alpha+t'^\alpha)
  \leq 2 M_5(3|t-t'|)^\alpha.
\]

If $t\leq t'\leq 2t\leq 1/2$, we already have
\[ |g(t,y_0)-g(t',y_0)| \leq M_5 |t-t'|^\alpha.
\]
The H\"older continuity of $g$ follows.

Combining these pieces of information, we conclude that
\[  g\in C^{1+\alpha}_\#(Q_3), \]
QED.

\section{Applications}
\subsection{Method of continuity}

The principle of the method of continuity consists in solving a
problem (P) by embedding it into a one-parameter family (P$_t$) of
problems, such that  (P$_0$) admits a unique solution, and (P$_1$)
coincides with problem (P). One then proves that the set of
parameter values for which (P$_t$) admits a unique solution is
both open and closed in $[0,1]$. The openness usually follows from
the implicit function theorem in H\"older spaces, and the
closedness from Ascoli's theorem; thus, both steps are made
possible by Schauder estimates.

We give an example in which a simplified procedure based on the
contraction mapping principle suffices. \btl{2} Let $L$ be an
elliptic operator with $C^\alpha$ coefficients and $c\leq 0$, in a
bounded domain $\Omega$ of class $C^{2+\alpha}$. Then, for any
$g\in C^{2+\alpha}(\overline\Omega)$, $Lu=f$ admits a solution in
$C^{2+\alpha}(\overline\Omega)$ which is equal to $g$ on
$\pa\Omega$. \etl
\begin{proof}
Considering $u-g$, we may restrict our attention to the case
$g=0$. We let $L_tu=tLu+(1-t)\Delta u$ and consider the problem
(P$_t$) which consists in solving $L_tu=f$ with Dirichlet
conditions. $L_t$ is a bounded operator from
$C^{2+\alpha}(\overline\Omega)\cup\{u=0\text{ on }\pa\Omega\}$ to
$C^\alpha(\overline\Omega)$. We know that $L_0$ is invertible, and
we wish to invert $L_1$. By the maximum principle, the assumption
$c\leq 0$ ensures that any solution of  (P$_t$) satisfies
$\sup_x|u(x)|\leq C\sup_x |f(x)|$, with a constant $C$ independent
of $t$. Therefore, if $T$ is the set of $t$ such that $L_t$ is
invertible, the Schauder estimates show that $L_t^{-1}$ is
bounded, and that its norm admits a bound $m$ independent of $t$.
This fact makes the rest of the proof simpler: let $t\in T$; for any
$s$, the equation $L_tu=f$ is equivalent to $u=L_t^{-1}f+M(t,s)u$
where
\[ M(s,t)u=(s-t)L_t^{-1}(L_0-L_1)u.
\]
If $|t-s|<\delta:=[m(\|L_0\|+\|L_1\|)]^{-1}$, $M(t,s)$ is a
contraction, and (P$_s$) is uniquely solvable. Covering $[0,1]$ by
a finite number of open intervals of length $\delta$, we find that
$L_t$ is invertible for every $t$. The result follows.
\end{proof}

For a typical example of the application of the method of
continuity, see \cite[th.~7.14]{Aubin}.

\subsection{Basic fixed-point theorems for compact operators}

We prove several versions of the Schauder fixed-point theorem. The
first ingredient in the proofs is the Brouwer fixed-point theorem:
\btl{brouwer} A continuous mapping $g : B\longrightarrow B$, where
$B$ is the closed unit ball in $\RR^n$, has at least one fixed
point. \etl
\begin{proof}
We begin with the case of smooth $g$. Assume that $g$ has no fixed
point. Let $\tilde x=x+a(x-g(x))$, where $a$ is the largest root
of the (quadratic) equation $|\tilde x|^2=1$. The point $\tilde x$
is on the intersection of the segment $[x,g(x)]$ with the unit
sphere, and is chosen so that $x$ lies between $\tilde x$ and
$g(x)$. The map from $B$ to its boundary defined by $x\mapsto
\tilde x$ is well-defined and smooth; in fact,
\[0=|\tilde x|^2-1=|x-g(x)|^2a^2+2(x,x-g(x))a+|x|^2-1,
\]
where $(\;,\;)$ denotes the usual dot product. The discriminant of
this quadratic is $4[(x,x-g(x))^2+(1-|x|^2)|x-g(x)|^2]$, which is
nonnegative, and vanishes only if $|x|=1$ and $(x,g(x))=1$. Since
$g(x)$ has norm one at most, the Cauchy-Schwarz inequality implies
that $g(x)=x$, which contradicts the hypothesis. Therefore, our
quadratic equation has two distinct real roots, obviously smooth.

For $|x|=1$, we find that $a=0$, since $(x,x-g(x))\geq 0$.

Define $f : \RR\times B\longrightarrow \RR^n$ by
\[ f(t,x_1,\dots,x_n)=x+ta(x)(x-g(x)).
\]
We find by inspection that (i) if $|x|=1$, $f(t,x)=x$ and
$\pa_tf(t,x)=0$; (ii) $f(0,x)=x$ for every $x$ in $B$; (iii)
$|f(1,x)|=1$ for every $x$ in $B$ (by construction of $a$).

Write $x_0$ for $t$, and define the determinants
\[ D_i=\det (f_{x_0},\dots,\hat f_{x_i},\dots,f_{x_n}),
\]
where $i$ runs from $0$ to $n$; a hat indicates that the
corresponding vector is omitted, and the subscripts denote
derivatives. Define further
\[ I(t)=\int_B D_0(t,x)dx.
\]
We have $I(0)=1$ since $f(0,x)=x$. For $t=1$, since $f$ lies on
the boundary of the unit sphere, $f_{x_1}$,\dots, $f_{x_n}$ are
all tangent to the sphere, and are linearly dependent; therefore,
$I(1)=0$.

We prove that $I(t)$ is constant, which will generate a
contradiction to the hypothesis that $g$ has no fixed point. We
need the
\begin{lemma}
$\sum_{i=0}^n(-1)^i\pa_{x_i}D_i=0$.
\end{lemma}
\begin{proof}
We have, for every $i$,
\[ \pa_{x_i}D_i=\sum_{j<i}(-1)^jC_{ij}+\sum_{j>i}(-1)^{j-1}C_{ij},
\]
where
\[ C_{ij}=\det (f_{x_ix_j},f_{x_0},\dots,\hat f_{x_i},\dots,\hat
f_{x_j},\dots,f_{x_n})=C_{ji}.
\]
Therefore
$\sum_{i=0}^n(-1)^i\pa_{x_i}D_i=\sum_{i,j=0}^n(-1)^{i+j}C_{ij}\sigma_{ij}$,
where $\sigma_{ij}=1$ for $j<i$, $-1$ for $j>i$, and zero for
$i=j$. Since $(-1)^{i+j}C_{ij}$ is symmetric in $i$ and $j$, and
$\sigma_{ij}$ is antisymmetric, the result follows.
\end{proof}
Now, for $i>0$, $D_i$ vanishes on the boundary of $B$ because
$\pa_t f=0$ there. If $n_i$ is the $i$-th component of the outward
normal to $B$, we find
\[ \int_B  \pa_{x_i}D_idx=\int_{\pa B}n_iD_ids=0.
\]
(This may be proved without using Stokes' theorem, by integrating
with respect to the $x_i$ variable keeping the others fixed.)
Using the lemma, we find
\[ \frac{dI(t)}{dt}=\int_B\pa_t D_0dx=\sum_{i>0}\pm
\pa_{x_i}D_idx=0.
\]
This completes the proof in the smooth case.

Finally, we extend the result to the case of continuous $g$. By
the Stone-Weierstrass theorem, there is a sequence of polynomial
(vector-valued) mappings $p_n$ such that $|g-p_n|\leq \ep_n\to 0$
uniformly over $B$. Since $p_n/(1+\ep_n)$ maps $B$ to itself,
there is a $y_n$ such that $p_n(y_n)=(1+\ep_n)y_n$. Extracting a
subsequence, we may assume $y_n$ has a limit $y$. It follows that
$g(y)=y$, QED.
\end{proof}
\vskip 1em
The Brouwer fixed-point theorem may be extended as follows:
\bt Let $K$ be the closed convex hull of a set of $N$ vectors
$x_1,\dots,x_N$ in $n$-dimensional space. A continuous map from
$K$ to itself has a fixed point. \etl
\begin{proof}
Let $\bar x=\frac 1N\sum_k x_k$. Decreasing $n$ if necessary, and
re-labeling the $x_k$, we may assume that the $(x_k-\bar x)_{k\leq
n}$ generate $\RR^n$. We prove that $K$ is homeomorphic to the
unit ball, so that the result follows from the Brouwer fixed point
theorem. First, $\bar x$ is interior to $K$, because, $\bar
x+\sum_{k\leq n}\ep_k(x_k-\bar x)$ is a convex combination of the
$x_k$ if the $\ep_k$ are small enough. Let $\ep$ be such that
$B_\ep(\bar x)\subset \mathop{\rm int}K$. Let, for any unit vector
$y$, $s(y)=\sup\{s : \bar x+sy\in L\}$. It is well-defined, and
bounded; also, $s(y)\geq\ep$. We need the following lemma.
\begin{lemma} $s(y)$ is continuous.
\end{lemma}
\begin{proof}
If $y_m\to y$ and $s(y_m)\to s$ as $m\to\infty$, with $\bar
x+s(y_m)y_m\in K$ for all $m$, we find $\bar x+ sy\in K$, hence
$s\leq s(y)$. If $s'<s(y)$, define $t=s'/s(y)\in[0,1]$, $(1-t)
B_\ep(\bar x)+ts(y)y$ is included in $K$ (which is convex), and is
a neighborhood of $\bar x+s'y$. This implies that  $\bar
x+s'y_m\in K$ for $m$ sufficiently large; it follows that
$s(y_m)\geq s'$ for $m$ large. Therefore, $s\geq s(y)$.
\end{proof}
\vskip 1em
We now construct the required homeomorphism from $B$ to $K$ by letting
$x\mapsto xs(x/|x|)$, which inverse $x\mapsto x/s(x/|x|)$. We just
proved that these maps are continuous at all points other than 0;
the continuity at the origin follows from the fact that $s$ and
$1/s$ are bounded.
\end{proof}
\vskip 1em
We now turn to fixed-point theorems in infinite dimensions.
\btl{fp1} If $K$ is a compact convex subset of a Banach space $E$,
and $T : K\longrightarrow K$ is continuity, then $T$ admits a
fixed point. \etl
\begin{proof}
For any integer $p$, there is an integer $N=N(p)$ and points
$x_1,\dots,x_N$ in $K$ such that $K\subset B(x_1,1/p)\cap\dots\cap
B(x_N,1/p)$. Let $B_k=B(x_k,1/p)$. Consider the closed convex hull
$K_p$ of $x_1,\dots,x_N$ which is a convex set which lies in some
finite-dimensional subspace of $E$; it is a subset of $K$. The map
\[ F_p : x\mapsto
\frac{\sum_k x_kd(x,K\setminus B_k)}{\sum_k d(x,K\setminus B_k)}
\]
is well-defined and continuous (the denominator does not vanish
because the $B_k$ cover $K$). Since any term on the numerator
contributes to the sum only if $|x-x_k|\leq 1/p$, we have
$\|F_p(x)-x\|_E\leq 1/p$.

The map $F_p\circ T$ therefore admits a fixed point $y_p$:
$F_p(T(y_p))=y_p$. We may extract a subsequence $y_{p'}$ which
tends to $y\in K$. We have $T(y_{p'})\to T(y)$, and
$\|F_{p'}(T(y_{p'}))-T(y_{p'})\|_E\to 0$. It follows that
$T(y)=y$.
\end{proof}
\btl{fp2} If $K$ is a closed convex subset of a Banach space $E$,
and $T : K\longrightarrow K$ is continuous, then, if $T(K)$ has
compact closure, then $T$ admits a fixed point. \etl
\begin{proof}
One approach would consist in working in the closure of the convex
hull of $T(K)$; this requires first proving that this set is
compact. A more direct argument is to apply the same method of
proof as in the previous theorem, with the difference that $K$ is
replaced by the closure of $T(K)$ in the definition of $F_p$. The
map $F_p\circ T$ is continuous on the closed convex hull of
$x_1,\dots, x_N$, and therefore has a fixed point $y_p$ as before.
We may extract a subsequence $y_{p'}$ such that $Ty_{p'}$ tends to
some $z$ in the closure of $T(K)$. Since
$\|F_{p'}(T(y_{p'}))-T(y_{p'})\|_E\to 0$, $y_{p'}$ also tends to
$z$. It follows that $Tz=z$.
\end{proof}
\vskip 1em
A useful variant is the following: \btl{fp3} Let $F$ be a
continuous mapping from the closed unit ball in a Banach space
$E$, with values in $E$ and with precompact image. If $\|x\|_E=1$
implies $\|T(x)\|_E<1$, then $T$ has a fixed point. \etl
\begin{proof}
It suffices to consider the mapping
\[S : x\mapsto T(x)/\max(1,\|T(x)\|_E),
\]
which is continuous with precompact image from the unit ball to
itself. It therefore possesses a fixed point $y$. If
$\|T(y)\|_E\geq 1$, we find that $y=T(y)/\|T(y)\|_E$ has norm 1;
the assumption now yields $\|T(y)\|_E<1$ : contradiction.
Therefore $\|T(y)\|_E<1$ and $T(y)=y$, QED.
\end{proof}
\vskip 1em
The next theorem asserts the existence of a fixed point
as soon as we have an \emph{a priori} bound.
Let $E$ denote a Banach space. Recall that a compact operator
is an operator which maps bounded sets to relatively compact sets.
\btl{schaeffer} Let $S : E\longrightarrow E$ be compact, and
assume that there is a $r>0$ such that if $u$ solves $u=\sigma
S(u)$ for some $\sigma\in [0,1]$, $\|u\|_E<r$. Then $S$
admits a fixed point in the ball of radius $r$ in $E$. \etl
\begin{proof}
Let $T(u)=S(u)$ if $\|S(u)\|_E\leq r$ and $T(u)=r S(u)/\|S(u)\|_E$
otherwise. Then the previous theorem applies and yields a fixed
point $u$ for $T$. If $\|S(u)\|_E\geq r$, $\|T(u)\|_E=r$ and
$u=T(u)=\sigma S(u)$, with $\sigma=r/\|S(u)\|_E\in[0,1]$.
Therefore, $\|u\|_E<r$. Since $u=T(u)$, we find $\|T(u)\|_E<r$,
which is impossible. Therefore, $\|S(u)\|_E< r$, and
$u=T(u)=S(u)$, QED.
\end{proof}
We note two useful variants: \btl{ls-global} Let $T : \RR\times
E\longrightarrow E$ be compact, and satisfy $T(0,u)=0$ for every
$u\in E$. Let $C_\pm$ denote the connected component of $(0,0)$ in
the set
\[ \{(\lambda,u)\in\RR\times E : u=T(\lambda,u)\text{ and }
\pm\lambda\geq 0\}.
\]
Then $C_+$ and $C_-$ are both unbounded. \etl
For this result, see \cite{LS,R}.
\btl{ls} Let $T :
[0,1]\times E\longrightarrow E$ be compact, and satisfy $T(0,u)=0$
for every $u\in E$. Assume that the relation $u=T(\sigma,u)$
implies $\|u\|_E<r$. Then equation $T(x,1)=x$ has a solution. \etl
\begin{proof}
Changing the norm on $E$, we may assume that $r=1$.

Let $\ep>0$, and consider the mapping $F_\ep$ defined by
\[ F_\ep(x)=T(\frac x{\|x\|_E},\frac {1-\|x\|_E}\ep)
\text{ if }1-\ep\leq \|x\|_E\leq 1,
\] and
\[ F_\ep(x)=T(\frac x{1-\ep},1)
\text{ if }\|x\|_E\leq 1-\ep,
\]
which is continuous with precompact image. Note that
\[ F_\ep(x)=T(\frac x{\max(1-\ep,\|x\|_E)},\min(1,\frac {1-\|x\|_E}\ep)).
\]
If $\|x\|_E= 1$, $F_\ep(x)=0$. Theorem \ref{th:fp2} applies, and
yields $x_\ep$ in the (open) unit ball such that
$F_\ep(x_\ep)=x_\ep$. For any integer $k\geq 1$, let
$y_p=x_{1/p}$, and $\sigma_p=\min(p(1-\|y_p\|_E),1)$. Since the
image of $T$ is precompact and the $\sigma_p$ are bounded, we may
extract a subsequence such that $(x_{p'},\sigma_{p'})$ tends to a
point $(x_\infty,\sigma_\infty)\in E\times [0,1]$.

If $\sigma_\infty<1$, all $\sigma_{p'}$ are less than 1 for large
$p'$, which means that $1-\|y_{p'}\|_E\geq 1/p'$. It follows that
$\|x_\infty\|=1$. The relation $x_\infty =
T(x_\infty,\sigma_\infty)$ now implies that $\|x_\infty\|<1$:
contradiction.

Therefore, $\sigma_\infty=1$. From the second expression for
$F_\ep$, it follows, by passing to the limit, that
$x_\infty=T(x_\infty,1)$, so that $x\mapsto T(x,1)$ has a fixed
point, QED.
\end{proof}

\subsection{Fixed-point theory and the Dirichlet problem}

We now apply the abstract theorems we just proved.

We begin with an application of theorem \ref{th:schaeffer}. Let
$\alpha$ and $\beta$ denote two numbers in $(0,1)$. Consider the
non-linear operator
\[ A : u\mapsto \sum_{ij}a^{ij}(x,u,\nabla
u)\pa_{ij}u+b(x,u,\nabla u),
\]
where $a^{ij}$ and $b$ are of class $C^\alpha$ in their arguments
say, globally, to fix ideas.\footnote{In many cases, the argument
below automatically yields \emph{a priori} bounds for $u$ and its
derivatives, so that one may truncate the nonlinearities for large
values of their arguments.} Let $g$ be a function of class
$C^{2+\alpha}(\overline\Omega)$. We wish to solve $Au=0$ in
$\Omega$, with $u=g$ on the boundary.

To $A$, we associate linear operators $A_v$, parameterized by a
function $v$:
\[ A_v : u\mapsto \sum_{ij}a^{ij}(x,v,\nabla
v)\pa_{ij}u+b(x,v,\nabla v),
\]
and an operator $T$ defined for $v\in
C^{1+\beta}(\overline\Omega)$, by $T(v)=u$, where $u$ is the
solution of the Dirichlet problem for equation
\[ A_v u =0
\]
in $\Omega$, with  $u=g$ on the boundary. Since $b(x,v,\nabla v)$
is easily seen to be of class $C^{\alpha\beta}$, the Schauder
estimates ensure that $u$ thus defined belongs to
$C^{2+\alpha\beta}(\overline\Omega)$. Note that $u=\sigma T(u)$
means that $\sum_{ij}a^{ij}(x,u,\nabla v)\pa_{ij}u+\sigma
b(x,u,\nabla u)$ in $\Omega$, and $u=\sigma g$ on the boundary.
\btl{4} If there is a $\beta\in(0,1)$ such that solutions in
$C^{2+\alpha\beta}$ of equation $A(u)=0$ in $\Omega$, with
$u=\sigma g$ on the boundary admit an \emph{a priori} bound of the
form $\|u\|_{C^{1+\beta}}\leq M$, with $M$ independent of $u$ and
$\sigma\in [0,1]$, then equation $A(u)=0$ admits at least one
solution with $u=g$ on the boundary. \etl
\begin{proof}
Operator $T$ maps bounded sets of $C^{1+\beta}$ to bounded sets of
$C^{2+\alpha\beta}$, which, by Ascoli's theorem, are relatively
compact in $C^{1+\beta}$. If $v_n\to v$ in $C^{1+\beta}$, the
functions $u_n=T(v_n)$ are bounded in $C^{2+\alpha\beta}$ by
Schauder estimates, and therefore, admit a convergent subsequence
$u_{n'}\to u$ in the $C^2$ topology, and \emph{a fortiori} in
$C^{1+\beta}$. Since
\[\sum_{ij}a^{ij}(x,v_n,\nabla
v_n)\pa_{ij}u_n+b(x,v_n,\nabla v_n)=0,
\]
it follows that $A_v(u)=0$. Therefore $T$ is continuous and
compact. The result now follows from theorem \ref{th:schaeffer}.
\end{proof}

We now turn to an application of theorem \ref{th:ls}, which arises
naturally if we wish $\sigma$ to enter in the definition of
$A_v$---which gives some flexibility in the perturbation argument.
We simply define $u=T(v,\sigma)$ by solving
\[ \sum_{ij}a^{ij}(x,v,\nabla
v,\sigma)\pa_{ij}u+b(x,v,\nabla v,\sigma)=0,
\]
with $u=\sigma g$ on the boundary. Here again, the existence of an
\emph{a priori} $C^{1+\beta}$ bound enables one to conclude that
$T(v,1)$ has a fixed point.

\subsection{Eigenfunctions and applications}

Since the inverse of the Laplacian (with Dirichlet boundary
condition) is compact, Riesz-Fredholm theory (see \cite{B})
ensures that the Laplacian admits a sequence of real eigenvalues
of finite multiplicity, tending to $+\infty$. The Fredholm
alternative holds: $\Delta u+\lambda u=f$ is solvable if and only
if $f$ is orthogonal to the eigenspace corresponding to the
eigenvalue $\lambda$.

We mention two important techniques related to Schauder theory:
bifurcation from a simple eigenvalue (see \cite{S72,Sm}, and the
Krein-Rutman theorem (see \cite{KR,R,Sm}).

\subsubsection{Bifurcation from a simple eigenvalue}
Consider, to fix ideas, the problem
\[ -\Delta u+\lambda u=u^2 \text{ on }\Omega,
\]
with Dirichlet boundary conditions. Assume we have an
eigenfunction $\phi_0$ for the simple eigenvalue $\lambda_0$:
\[ -\Delta \phi_0+\lambda \phi_0=0,
\]
with $\phi_0=0$ on the boundary. Let $Q[u]=\int_\Omega u\phi_0dx$
and $P[u]=u-\phi_0Q[u]$. We seek a family $(\mu(\ep),v(\ep))$ such
that our nonlinear problem admits the solutions $(\lambda,u)$,
where
\[ u=\ep\phi_0+\ep^2v(\ep);\qquad \lambda=\lambda_0+\ep\mu(\ep).
\]
In other words, we seek a curve of solutions which is tangent to
the eigenspace for the eigenvalue $\lambda_0$. If $\ep\mu$ is
small, it is easy to see that $-\Delta+\lambda$ is invertible on
the orthogonal complement of this eigenspace. Projecting on the
orthogonal complement of $\phi_0$, we find
\[ v=(-\Delta+\lambda)^{-1}P[(\phi_0+\ep v)^2],
\]
which may be solved for $v$ as a function of $\mu$, by the
implicit function theorem. This gives a map $v=\Psi[\ep,\mu]$.
Projecting the equation on $\phi_0$ now yields an equation for
$\mu(\ep)$:
\[ \mu(\ep)=Q[(\phi_0+\ep \Psi[\ep,\mu])^2],
\]
which may be solved for $\mu(\ep)$, again by an implicit function
theorem. We find $\mu(\ep)=Q[\phi_0^2]+O(\ep)$. For variants of
this argument, see \emph{e.g.} \cite[Ch.~5]{nlw}.

\subsubsection{Krein-Rutman theorem}

We wish to generalize to infinite dimensions a classical property
of matrices with nonnegative entries.

We first need a variant of theorem \ref{th:ls-global}, which
follows from it using an extension theorem due to Dugundji (see
\cite{Dugundji,R,Sm}).
\btl{cone} Let $K$ be a closed convex cone
with vertex 0, and let $T : \RR^+\times K\longrightarrow K$ be
compact, and assume $T(0,u)=0$ for every $u$. Then the connected
component of $(0,0)$ in the set of all solutions $(\lambda,u)$ of
$u=T(\lambda,u)$ is unbounded.
\etl
As a consequence, we derive
the ``compression of a cone'' theorem:
\btl{comp-cone} Let $K$ be
a closed convex cone with vertex $0$ and non-empty interior, with
the property
\[ K\cap(-K)=\{0\}.
\]
Let $L$ denote a bounded linear operator on $E$ which maps
$K\setminus\{0\}$ to the interior of $K$. Then there is a unit
vector in $K$ and a positive real $\mu$ such that $Lx_0=\mu$.
\etl
\begin{remark}
A typical application: let $E=C^{1+\alpha}(\Omega)$, with $\Omega$
bounded and smooth, take for $L$ the inverse of an elliptic
operator, such as $-\Delta +c(x)$, with $c\geq 0$, and for $K$ the
closure of $\{u\in E : u>0\text{ in }\Omega,\text{ and } \pa u/\pa
n<0 \text{ on }\pa\Omega\}$, where $\pa/\pa n$ denotes the outward
normal derivative. As usual, the compactness is ensured by the
Schauder estimates. The fact that $L$ is a ``compression'' of the
cone $K$, \emph{i.e.} sends $K\setminus\{0\}$ to the interior of
$K$, follows from the Hopf maximum principle. Note that the
conclusion $x_0\in K$ gives directly the information that the
first eigenfunction is positive throughout $\Omega$.
\end{remark}
\begin{proof}
In this proof only, we write $u\geq v$ when $u-v\in K$. Fix $u\in
K\setminus\{0\}$; in particular, $Lu$, which is interior to $K$,
cannot be equal to 0. There is a positive $M$ such that $Lu\geq
u/M$, for otherwise, we would have $Lu-u/M\not\in K$ for all
$M>0$, and, letting $M\to\infty$, we would find $Lu\not\in
\mathop{\rm int}K$.

For any $\ep>0$, consider the compact operator defined by
$T_\ep(\lambda,x)=\lambda L(x+\ep u)$. Let $C_\ep$ be the
connected component of $(0,0)$ in $\RR_+\times K$ of the set of
solutions of $x=T_\ep(\lambda,x)$; we know that it is unbounded.
For such a solution, we have, since $x\in K$, $x=\lambda
Lx+\lambda\ep u\geq \lambda\ep u$. Since $K$ is invariant under
$L$, we find $Lx\geq\lambda\ep Lu\geq \lambda\ep u/M$.
We also have $x\geq \lambda Lx$; therefore,
$x\geq \lambda^2\ep u/M$, and $Lx\geq
(\lambda/M)^2\ep u$. By induction, we find $Lx\geq
(\lambda/M)^n\ep u$ for every $n\geq 1$. If $\lambda>M$, we find,
letting $n\to\infty$, that $\ep u\leq 0$, which means $u\in -K$.
Since $u\in K$, and $u\neq 0$, this is impossible.
Therefore, $C_\ep$ lies in
$[0,M]\times K$. Since $C_\ep$ is unbounded and contains $(0,0)$,
there is, for every $\ep>0$, a unit vector $x_\ep\in
K$ such that
\[ x_\ep=\lambda_\ep L(x_\ep+\ep u)\text{ and
}0\leq\lambda_\ep\leq M.
\]
Since $L$ is compact, there is a sequence $\ep_n\to 0$ and a
$(\mu,x_0)\in [0,M]\times K$ such that $x_{\ep_n}\to x_0$ and
$\lambda_{\ep_n}\to\mu$. It follows that $x_0=\mu Lx_0$ and
$\|x_0\|_E=1$. Since $x_0\neq 0$, we must have $\mu>0$ and also
$x_0\in \mathop{\rm int} K$. This completes the proof.
\end{proof}

\subsection{Method of sub- and super-solutions}

Consider the problem
\begin{equation}\label{eq:ssol} -\Delta u = f(u)
\end{equation}
with Dirichlet boundary conditions on a smooth bounded domain
$\Omega$, and
$f$ smooth, such that $f$ and $df/du$ are both
bounded.\footnote{The boundedness condition is not as restrictive
as it seems: for instance, if $u$ represents a concentration, it
must lie between 0 and 1, and $f$ may be redefined outside $[0,1]$
so that it is bounded.} We assume that we are given two
\emph{ordered sub- and super-solutions} $v$ and $w$: $v$ and $w$
are of class $C^2(\overline\Omega)$, vanish on $\pa\Omega$ and
satisfy, over $\Omega$,
\[ v\leq w;\quad-\Delta v\leq f(v);\quad -\Delta w\geq f(w).
\]
We then have:
\btl{ssol} Problem (\ref{eq:ssol}) admits two
solutions $\underline u$ and $\overline u$ such that
\[ v\leq \underline u \leq \overline u \leq w.
\]
In addition, if $u$ is any solution of (\ref{eq:ssol}) which lies
between $v$ and $w$, then necessarily $\underline u \leq u\leq
\overline u$.
\etl
\begin{remark}
For more results of this kind, see \emph{e.g.} \cite{S,S72}.
\end{remark}
\begin{proof}
Choose a constant $m$ such that $g(u):=f(u)+mu$ is strictly
increasing. Define inductively two sequences $(v_j)_{j\geq 0}$ and
$(w_j)_{j\geq 0}$ by the relations: $v_0=v$; $w_0=w$;
\[ -\Delta v_j+mv_j=g(v_{j-1});\qquad
   -\Delta w_j+mw_j=g(w_{j-1})\text{ for }j\geq 1,
\]
and $v_j=w_j=0$ on $\pa\Omega$. We have $(-\Delta +
m)(v_1-v_0)\geq g(v_0)-g(v_0)=0$, which implies $v_1\geq v_0$ by
the maximum principle.\footnote{See \emph{e.g.} \cite{B} for a
simple proof.} Since $(-\Delta
+m)(v_{j+1}-v_j)=g(v_j)-g(v_{j-1})$, we find by induction
$(-\Delta +m)(v_{j+1}-v_j)\geq 0$, hence $v_{j+1}-v_j\geq 0$.
Therefore, the sequence $(v_j)$ is non-decreasing. Similarly,
$(w_j)$ is non-increasing. In addition, $(-\Delta
+m)(w_0-v_0)=g(w_0)-g(v_0)\geq 0$, and $(-\Delta
+m)(w_j-v_j)=g(w_{j-1})-g(v_{j-1})$ for $j\geq 1$. It follows that
$w_0\geq v_0$ and, by induction, $w_j\geq v_j$. We conclude that
$\underline u:=\lim_{j\to\infty}v_j$ and $\overline
u:=\lim_{j\to\infty}w_j$ exist and satisfy
\[ v_0\leq v_1\leq\cdots\underline u
\leq\overline u\leq\cdots w_1\leq w_0.
\]
By construction, the $v_j$ are bounded. Therefore, $(-\Delta
+m)v_j$ is bounded independently of $j$. Consider now any ball
$B_r$ such that $\overline B_{2r}\subset\Omega$, and fix
$\alpha\in(0,1)$. The interior $C^{1+\alpha}$ Schauder estimates
ensure first that the $v_j$ are, for $j\geq 1$ bounded in $C^1(B_{3r/2})$,
independently of $j$. This implies in particular
a $C^\alpha$ bound on $g(v_j)$. The $C^{2+\alpha}$ Schauder
estimates now ensure that the $v_j$ are bounded in $C^2(B_r)$ for $j\geq
1$, and that their second derivatives are
equicontinuous. It follows that one may extract
a subsequence $v_{j'}$ which converges to $\underline u$
in $C^2(B_r)$. It follows that $(-\Delta
+m)\underline u=f(\underline u)+m\underline u$; so that
$\underline u$ solves (\ref{eq:ssol}). A similar argument applies
to $\overline u$. Finally, if $u$ is a solution such that $v_0\leq
u\leq w_0$, we have $(-\Delta+m)(v_0-u)\leq g(v_0)-g(u)$ and
$(-\Delta+m)(v_j-u)=g(v_{j-1})-g(u)$ for $j\geq 1$. It follows, by
induction, that $v_j\leq u$ for all $j$. Similarly, $w_j\geq u$
for all $j$. Passing to the limit, we find $\underline u\leq u\leq
\overline u$.
\end{proof}

\subsection{Asymptotics near isolated singularities or at infinity}

We give three simple examples where Schauder estimates help
understand the behavior of solutions at infinity or at isolated
singularities.

\subsubsection{Liouville property} Regularity theory gives a
simple proof of the Liouville property for scale-invariant
equations. Consider for instance the $p$-Laplace equation
$A_pu:=\mathop{\rm div}(|\nabla u|^{p-2}\nabla u)=0$,
where $p>1$. We have \cite{SK-th} an interior
$C^1$ estimate of the form
\[ \|u\|_{C^1(B_1)}\leq C \sup_{B_2}|u|.
\]
Applying it to $u(Rx)$, we find, since $\nabla (u(Rx))=R(\nabla
u)(Rx)$,
\[ \sup_{B_R}|\nabla u|\leq \frac CR \sup_{B_{2R}}|u|.
\]
Letting $R\to\infty$, it follows immediately that any solution
which is bounded on all of $\RR^n$ is constant. A more subtle
result of this type is: any nonnegative solution on
$\RR^n\setminus 0$ is necessarily constant \cite[p.~602]{KV}.

\subsubsection{Asymptotics at infinity} If $u$ solves $Lu=f$ on an
exterior domain $\{|x|>\rho\}$, where the coefficients of $L$ tend
to constants at infinity, one may hope to apply the above scaling
argument on balls $B_R(x_R)$, where, say, $|x_R|\geq 2R\to\infty$.
In this way, it is possible to obtain weighted estimates at
infinity, which are useful in solving the constraints equations in
General Relativity \cite{cc} or in asymptotics for solutions of the
Ginzburg-Landau equation \cite{pr}.

\subsubsection{Asymptotics near isolated singularities}
\label{sec:isol-sing}

The $C^{1+\alpha}$
Schauder-type estimates for the  $p$-Laplace equation $A_pu=0$
may be used to determine the behavior
at the origin of positive solutions in a
punctured neighborhood of the origin. For instance, if $n\geq 2$
and $p<n$ and
\[\mu(r)=\frac{p-1}{n-p}(n\omega_n)^{-1/(p-1)}r^{(p-n)/(p-1)},
\]
resp.\ $\mu(r)=(n\omega_n)^{-1/(n-1)}\ln(1/r)$ for $p=n$,
then any solution which is bounded above and below by positive
multiples of $\mu(|x|)$ must in fact be of the form $\gamma
\mu(|x|)+O(1)$ for some constant $\gamma$. In fact,
\[ -A_p u =\gamma|\gamma|^{p-2}\delta_0,
\]
in the sense of distributions, where $\delta_0$ is the Dirac
distribution at the origin.

Regularity estimates
enter the argument as follows: to consider the family of functions
$u_r(y)=u(ry)/\mu(r)$, which, by Schauder-type $C^{1+\alpha}$
estimates, satisfies a compactness condition on annular domains.
Letting $r\to 0$ along a suitable sequence, we find that $u_r$
tends to a solution $v$ of $A_pv=0$ outside the origin, and we may
arrange so that $v(y)/\mu(|y|)$ has an interior maximum $\gamma$.
At such a maximum, the gradient of $v$ is proportional to the
gradient of $\mu$ and thus does not vanish, so that the equation
is in fact uniformly elliptic near the point of maximum; this
makes it possible to conclude that $v/\mu$ is in fact constant,
using the strong maximum principle (as pointed out in
\cite[p.~263]{GT}, the difference $w=u-\gamma \mu$ solves a linear
elliptic equation). See \cite{SK-th,KV} for details and further
results. For $p=n$, one can see that $u-\gamma \mu$ has a limit at
the origin; this fact has found recent applications
\cite{bht,colesanti}. For similar results for semilinear
equations, see \cite{gs,cgs}.

\subsection{Asymptotics for boundary blow-up}

We give a typical application of Fuchsian reduction to elliptic
problems \cite{SK-ln,SK-gaeta}.
The proof structure hinges on general properties of the
Fuchsian Reduction process and is therefore liable of application
to many other situations.

\subsubsection{Main result and structure of proof}
Let $\Omega\subset\RR^n$, $n\geq 3$, be a bounded domain of class
$C^{2+\alpha}$, where $0<\alpha<1$. Consider the Loewner-Nirenberg
equation in the form
\begin{equation}\label{eq:LN}
-\Delta u + n(n-2)u^{\frac{n+2}{n-2}}=0.
\end{equation}
It is known \cite{LN,BF,BM,BMd,MV} that this equation admits a
maximal solution $u_\Omega$, which is positive and smooth inside
$\Omega$; it is the limit of the increasing sequence $(u_m)_{m\geq
1}$ of solutions of (\ref{eq:LN}) which are equal to $m$ on the
boundary. It arises in many contexts \cite{BF,LN}. We note for later
reference the monotonicity property: if
   $\Omega\subset\Omega'$, then any classical solution in $\Omega'$
   restricts to a classical solution in $\Omega$, so that
   \begin{equation}\label{eq:monot}
   u_{\Omega'}\leq u_\Omega;
  \end{equation}
it follows easily from the maximality of $u_\Omega$.
The \emph{hyperbolic radius} of $\Omega$ is the function
\[
v_\Omega:=u_\Omega^{-2/(n-2)};
\]
it vanishes on $\pa\Omega$. Let $d(x)$ denote the distance of $x$
to $\pa\Omega$. It is of class $C^{2+\alpha}$ near $\pa\Omega$. We
prove
\begin{theorem}\label{th:1}
If $\Omega$ is of class $C^{2+\alpha}$, then $v_\Omega\in
C^{2+\alpha}(\overline\Omega)$, and
\[ v_\Omega(x) = 2d(x)-d(x)^2[H(x)+o(1)]
\]
as $d(x)\to 0$, where $H(x)$ is the mean curvature at the point of
$\pa\Omega$ closest to $x$.
\end{theorem}
This result is optimal, since $H$ is of class $C^\alpha$ on the
boundary. It follows from theorem \ref{th:1} that $v_\Omega$ is a
\emph{classical solution} of
\[v_\Omega\Delta v_\Omega=\frac{n}2(|\nabla v_\Omega|^2-4),
\]
even though $u_\Omega$ cannot be interpreted as a weak solution of
(\ref{eq:LN}), insofar as $u_\Omega^{\frac{n+2}{n-2}}\sim
(2d)^{-1-n/2}\not\in L^1(\Omega)$.

We now give an idea of the proof.

We begin by performing a Fuchsian reduction, that is, we introduce
the degenerate equation solved by a renormalized unknown, which
governs the higher-order asymptotics of the solution; in this
case, a convenient renormalized unknown is
\[
w:=(v_\Omega-2d)/d^2.
\]
It follows from general arguments, see the overview in
\cite{SK-gaeta,SK-fr}, that the equation for $w$ has a very
special structure: the coefficient of the derivatives of order $k$
is divisible by $d^k$ for $k=0$, 1 and 2, and the nonlinear terms
all contain a factor of $d$. Such an equation is said to be
\emph{Fuchsian}.

In the present case, one finds
\begin{equation}\label{eq:fr}
\frac {2v^{n/2}}{n-2}\{ -\Delta
u_\Omega+n(n-2)u_\Omega^{(n+2)/(n-2)}\} =Lw+2\Delta d-M_w(w),
\end{equation}
where
\[
L:=d^2\Delta+(4-n)d\nabla d\cdot\nabla+(2-2n),
\]
and $M_w$ is a linear operator with $w$-dependent coefficients,
defined by
\[
M_w(f):=\frac{nd^2}{2(2+dw)} [2f\nabla d\cdot\nabla w+d\nabla
w\cdot\nabla f]-2df\Delta d.
\]
The proof now consists in a careful bootstrap argument in which
better and better information on $w$ results in better and better
properties of the degenerate linear operator $L-M_w$. A key step
is the inversion of the analogue of $L$ in the half-space, which
plays the role of the Laplacian in the usual Schauder theory.

Equation (\ref{eq:fr}) needs only to be studied in the
neighborhood of the boundary. Let us therefore introduce
$C^{2+\alpha}$ thin domains $\Omega_\delta=\{0<d<\delta\}$, such
that $d\in C^{2+\alpha}(\overline\Omega_\delta)$, and
$\pa\Omega_\delta=\pa\Omega\cup\Gamma$ consists of two
hypersurfaces of class $C^{2+\alpha}$.

Recall that
\[ \|u\|_{C_\#^{k+\alpha}(\overline\Omega_\delta)}
:=\sum_{j=0}^k \|d^ju\|_{C^{j+\alpha}(\overline\Omega_\delta)}.
\]

The proof proceeds in five steps, corresponding to five
theorems: first, a comparison argument combined with
Schauder estimates gives
\begin{theorem}\label{th:I}
$w$ and $d^2\nabla w$ are bounded near $\pa\Omega$.
\end{theorem}
Theorem \ref{th:I} ensures that $L-M_w$ is of type (I). Theorem
\ref{th:FIa} then implies that $d\nabla w$ is bounded near the boundary;
going back to the definition of $M_w$, we find $M_w(w)=O(d)$; this
yields the next theorem:
\begin{theorem}\label{th:II}
$d\nabla w$ and $M_w(w)/d$ are bounded near $\pa\Omega$.
\end{theorem}
At this stage, we have $Lw+2\Delta w=O(d)$. In order to use theorem
\ref{th:FIa}, we need to subtract from $w$ a function $w_0$ such
that $Lw_0+2\Delta=0$ with controlled boundary behavior, and $w-w_0=O(d)$;
the function $w_0$ is constructed in:
\begin{theorem}\label{th:III}
If $\delta$ is sufficiently small, there is a $w_0\in
C_\#^{2+\alpha}(\overline\Omega_\delta)$ such that
\begin{equation}\label{eq:Lw0}
Lw_0+2\Delta d=0
\end{equation}
in $\Omega_\delta$. Furthermore
\begin{equation}\label{eq:H}
w_0\,\rest{\pa\Omega}=-H,
\end{equation}
where $H=-(\Delta d)/(n-1)$ is the mean curvature of the boundary.
\end{theorem}
Incidentally, we see how the curvature of the boundary enters into the
asymptotics. We now use a comparison function of the form
$w_0+Ad$, where $A$ is a constant, to bound $w-w_0$:
\begin{theorem}\label{th:IV}
Near the boundary,
\[
\tilde{w}:=w-w_0=O(d).
\]
\end{theorem}
At this stage, we know that
\[ L\tilde w=O(d)\text{ and }\tilde w=O(d)\]
near $\pa\Omega$. Theorem \ref{th:FIb} yields that $\tilde w$ is in
$C_\#^{1+\alpha}(\overline\Omega_\delta)$, for $\delta$ small
enough. It follows that $M_w(w)\in
C^{\alpha}(\overline\Omega_\delta)$. We may now use theorem
\ref{th:FIIa} to conclude that $d^2w$ is of class $C^{2+\alpha}$
near the boundary. Since
$\tilde{w}=O(d)$, $w\,\rest{\pa\Omega}$ is equal to $-H$. This
completes the proof of theorem \ref{th:1}.

We write henceforth $u$ and $v$ for $u_\Omega$ and $v_\Omega$
respectively. The rest of this section is devoted to the proofs of
the above theorems.

It remains to prove theorems \ref{th:I}, \ref{th:III} and
\ref{th:IV}.

Theorem \ref{th:I} is proved in section \ref{sec:comp2} by a
comparison argument combined with regularity estimates, as in
section \ref{sec:isol-sing}.

Theorem \ref{th:III} is proved in three steps: first, one
decomposes $L$ into a sum $L_0+L_1$ in a coordinate system adapted
to the boundary, where $L_0$ is the analogue of $L$ in a
half-space in the new coordinates (section \ref{sec:IIIa}); next,
one solves $Lf=g+O(d^\alpha)$ in this coordinate system
for any function of class
$C^\alpha$---such as $-2\Delta d$---by inverting a model operator closely
related to $L_0$ (section \ref{sec:IIIb}); finally, we patch the
results to obtain a function $w_0$ such that $Lw_0=g$ (section
\ref{sec:consL}).

Theorem \ref{th:IV} is proved in section \ref{sec:comp2} by a
second comparison argument.

\subsection{First comparison argument}
\label{sec:comp1}

Since $\pa\Omega$ is $C^{2+\alpha}$, it satisfies a uniform
interior and exterior sphere condition, and there is a positive
$r_0$ such that any $P\in\Omega$ such that $d(P)\leq r_0$ admits a
unique nearest point $Q$ on the boundary, and such that there are
two points $C$ and $C'$ on the line determined by $P$ and $Q$,
such that
\[B_{r_0}(C)\subset\Omega\subset\RR^n\setminus B_{r_0}(C'),\]
these two balls being tangent to $\pa\Omega$ at $Q$. We now define
two functions $u_i$ and $u_e$. Let
\[u_i(M)=(r_0-\frac{CM^2}{r_0})^{1-n/2}\text{ and }
u_e(M)=(\frac{C'M^2}{r_0}-r_0)^{1-n/2}.
\]
$u_i$ and $u_e$ are solutions of equation (\ref{eq:LN}) in
$B_{r_0}(C)$ and $\RR\setminus B_{r_0}(C')$ respectively.

If we replace $r_0$ by $r_0-\varepsilon$ in the definition of
$u_e$, we obtain a classical solution of (\ref{eq:LN}) in
$\Omega$, which is therefore dominated by $u_\Omega$. It follows
that
\[ u_e\leq u_\Omega\text{ in }\Omega.\]
The monotonicity property (\ref{eq:monot}) yields
\[ u_\Omega\leq u_i\text{ in }B_{r_0}(C).\]

In particular, the inequality
\[ u_e(M)\leq u_\Omega(M)\leq u_i(M)\]
holds if $M$ lies on the semi-open segment $[P,Q)$. Since $Q$ is
then also the point of the boundary closest to $M$, we have
$QM=d(M)$, $CM=r_0-d$ and $C'M=r_0+d$; it follows that
\[(2d+\frac{d^2}{r_0})^{1-n/2}\leq u_\Omega(M)
\leq (2d-\frac{d^2}{r_0})^{1-n/2}.
\]
Since $u_\Omega=(2d+d^2w)^{1-n/2}$, it follows that
\[ |w|\leq \frac1{r_0}\text{ if }d\leq r_0.  \]

Next, consider $P\in\Omega$ such that $d(P)=2\sigma$, with
$3\sigma<r_0$. For $x$ in the closed unit ball $\overline B_1$,
let
\[  P_\sigma :=P+\sigma x;\quad
    u_\sigma(x) := \sigma^{(n-2)/2}u(P_\sigma).
\]
One checks that $u_\sigma$ is a classical solution of
(\ref{eq:LN}) in $\overline B_1$. Since $d\mapsto 2d\pm \frac
1{r_0}d^2$ is increasing for $d<r_0$, and $d(P_\sigma)$ varies
between $\sigma$ and $3\sigma$ if $x$ varies in $\overline B_1$,
we have
\[(6+\frac{9\sigma}{r_0})^{1-n/2}\leq u_\sigma(M)
\leq (2-\frac{\sigma}{r_0})^{1-n/2}.
\]
This provides a uniform bound for $u_\sigma$ on $B_1$. Applying
interior regularity estimates as in \cite{SK-th,BM}, we find that
$\nabla u_\sigma$ is uniformly bounded for $x=0$. Recalling that
$\sigma=\frac12 d(P)$, we find that
\[ d^{\frac n2-1}u\text{ and }d^{\frac n2}\nabla u\text{ are
bounded near }\pa\Omega.
\]
It follows that $u^{-n/(n-2)}=O(d^{n/2})$, and since
$d^2w=-2d+u^{-2/(n-2)}$, we have
\[ d^2\nabla w=-2(1+dw)\nabla d-\frac 2{n-2}u^{-n/(n-2)}\nabla u,
\]
hence $d^2\nabla w$ is bounded near $\pa\Omega$. This completes
the proof of theorem \ref{th:I}.

\subsubsection{Decomposition of $L$ in adapted coordinates}
\label{sec:IIIa}

Since $\pa\Omega$ is compact, there is a positive $r_0$ such that
in any ball of radius $r_0$ centered at a point of $\pa\Omega$,
one may introduce a coordinate system $(Y,T)$ in which $T=d$ is
the last coordinate. The formulae of section \ref{sec:prel} apply.
It will be convenient to assume that the domain of this coordinate
system contains a set of the form
\[ 0<T<\theta\text{ and }|Y_j|<\theta\text{ for }j\leq n-1.
\]
Let $\pa_j=\pa_{x_j}$, and write $d_n$ and $d_j$ for $\pa d/\pa
x_n$ and $\pa d/\pa x_j$ respectively. Primes denote derivatives
with respect to the $Y$ variables: $\pa'_j=\pa_{Y_j}$,
$\nabla'=\nabla_Y$, $\Delta'=\sum_{j<n}\pa_j^{\prime 2}$, etc.
We write $\tilde\nabla d=(d_1,\dots,d_{n-1})$.
Recall that $|\nabla d|=1$. We let throughout
\[
D=T\pa_T.
\]
The transformation formulae are
\begin{eqnarray*}
    T & =&  d(x_1,\dots,x_n); \qquad Y_j = x_j\text{ for }j<n;\\
    \pa_n &=& d_n\pa_T; \qquad \pa_j=d_j\pa_T+\pa_j'.
\end{eqnarray*}
We recall that $\Delta d=(1-n)H$, where $H$ is the mean curvature
of $\pa\Omega$.

We further have
\begin{eqnarray*}
    d\nabla d\cdot\nabla w &=& (D+T\tilde\nabla d\cdot\nabla')w \\
    |\nabla w|^2 &=& w_T^2+|\nabla'
    w|^2+2w_T\tilde\nabla d\cdot\nabla'w\\
    \Delta w &=&
    w_{TT}+\Delta'w+2\tilde\nabla d\cdot\nabla'w_T+w_T\Delta d.
\end{eqnarray*}
It follows that
\[ Lw=L_0w+L_1w,\]
where
\[ L_0w=(D+2)(D+1-n)w+T^2\Delta'w,
\]
and
\[ L_1w=(4-n)\tilde\nabla d\cdot\nabla'(Tw)
        +2T\tilde\nabla d\cdot\nabla'(Dw)+T(Dw)\Delta d.
\]

\subsubsection{Solution of $Lf=k+O(d^\alpha)$}
\label{sec:IIIb}

We now solve approximately equation $Lf=k$ by solving exactly a
model problem, related to the operator $L_0$.

Let $C^\alpha_{\text{per}}$ denote the space of functions
$k(Y,T)\in C^\alpha(0\leq T\leq \theta)$ which satisfy
$k(Y_j+2\theta,T)=k(Y_j,T)$ for $1\leq j\leq n-1$ . We prove the
following theorem.
\begin{theorem}\label{th:L0}
Let $\theta>0$, and $k(Y,T)$ of class $C^\alpha_{\text{per}}$ Then
there is a function $f$ such that
\begin{enumerate}
  \item $L_0f=k+O(d^\alpha)$,
  \item $f$ is of class $C_\#^{2+\alpha}(0\leq T\leq\theta)$,
  \item $f(Y,0)=k(Y,0)/(2-2n)$ and
  \item $L_1f=O(d^\alpha)$.
\end{enumerate}
\end{theorem}
\begin{proof}
Let
\[ L'_0=(D+2)(D-1)+T^2\Delta'=L_0+(n-2)(D+2).
\]
We first solve the equation $L'_0f_0=k$.
\begin{lemma} There is a bounded linear operator
\[ G : C^\alpha_{\text{per}}\longrightarrow C_\#^{2+\alpha}(0\leq
T\leq\theta)
\]
such that $f_0:=G[k]$ verifies
\begin{enumerate}
  \item $L'_0f_0=k$,
  \item $f_0$ is of class $C_\#^{2+\alpha}(0\leq T\leq\theta)$,
  \item $f_0(Y,0)+k(Y,0)/2=0$, $Df_0(Y,0)=0$ and
  \item $L_1f_0=O(d^\alpha)$.
\end{enumerate}
\end{lemma}
\begin{proof}
One first constructs $\tilde k$ such that $(D-1)\tilde k=-k$, and
$\tilde k$ and $D\tilde k$ are both $C^\alpha$ up to $T=0$. One
may take
\[ \tilde{k}=\int_1^\infty
F_1[k](Y,T\sigma)\frac{d\sigma}{\sigma^2}.
\]
where $F_1$ is an extension operator, so that $F_1[k]=k$ for
$T\leq \theta$.

One checks that $\tilde k=k$ for $T=0$.

One then solves $(\pa_{TT}+\Delta')h+\tilde k=0$ with periodic
boundary conditions, of period $2\theta$, in each of the $Y_j$,
and $h(Y,0)=h_T(Y,\theta)=0$; this yields
\[ h \text{ is of class }C^{2+\alpha}(0\leq T\leq \theta)
\]
by the Schauder estimates. In particular, $h_T$ is continuous up
to $T=0$, and $Dh=0$ for $T=0$ and $T=\theta$.

Since $h=0$ for $T=0$, we also have $\Delta' h=0$ for $T=0$. The
equation for $h$ therefore gives
\[h_{TT}=-\tilde k=-k\text{ for }T=0.
\]

In addition,
\[(\pa_{TT}+\Delta')Dh=D(\pa_{TT}+\Delta')h+2h_{TT}=k-\tilde
k+2h_{TT},
\]
which is $C^\alpha$. Since, on the other hand, $Dh$ is of class
$C^1$ and $Dh=0$ for $T=0$ and $T=\theta$, we conclude, using
again the Schauder estimates, that
\[ Dh \text{ is of class }C^{2+\alpha}(0\leq T\leq \theta).
\]

We now define $f_0$ by
\begin{equation}\label{eq:G}
f_0:=T^{-2}(D-1)h
        =\pa_T\left(\frac h T\right)
        =\int_0^1 \sigma h_{TT}(Y,T\sigma)\,d\sigma.
\end{equation}
Since $f_0$ is itself uniquely determined by $h$, itself defined
in terms of $k$ we define a map $G$ by
\[f_0=G[k].
\]
A direct computation yields $L'_0f_0=k$:
\begin{align*}
L'_0f_0 &=(D+2)(D-1)T^{-2}(D-1)h+(D-1)\Delta' h\\
       &=T^{-2}D(D-3)(D-1)h+(D-1)\left\{-T^{-2}D(D-1)h-\tilde{k}\right\}\\
       &=T^{-2}D(D-1)(D-3)h-T^{-2}(D-3)D(D-1)h-(D-1)\tilde{k}\\
       &=k.
\end{align*}

Let us now consider the regularity of $f_0$ up to $\pa\Omega$, and
the values of $f_0$ and its derivatives on $\pa\Omega$.

Consider $g_0:=T^2f_0$. Since $g_0=(D-1)h\in C^{2+\alpha}(0\leq
T\leq \theta)$ and vanishes for $T=0$, we have $g_0=\int_0^1
g_{0T}(Y,T\sigma)Td\sigma$. It follows that
\[Tf_0(Y,T)=\int_0^1 g_{0T}(Y,T\sigma)d\sigma
\in C^{1+\alpha}(0\leq T\leq \theta).
\]
Since, on the other hand, $G[k]=\int_0^1 \sigma
h_{TT}(Y,T\sigma)\,d\sigma$, we find $f_0\in C^\alpha(0\leq T\leq
\theta)$, and
\[f_0(Y,0)=\frac12 h_{TT}(Y,0)=-\frac12 k(Y,0).
\]
We therefore have
\[f_0 \text{ is of class }C_\#^{2+\alpha}(0\leq T\leq \theta).
\]
Since
\[ (D+2)f_0 =T^{-2}D(D-1)h= h_{TT},
\]
we find $Df_0(Y,0)=h_{TT}(Y,0)-2f_0(Y,0)=0$. By differentiation
with respect to the $Y$ variables, we obtain that $\tilde\nabla
d\cdot\nabla'(Tf_0)$ is of class $C^\alpha$ and vanishes for
$T=0$. The same is true of $T(Df_0)\Delta d$. Similarly,
\[
2T\tilde\nabla d\cdot\nabla'Df_0 = 2\tilde\nabla
d\cdot\nabla'[\pa_T(T^2f_0)-2Tf_0]
\]
is of class $C^\alpha$, and vanishes for $T=0$ because this is
already the case for $TDf_0$. It follows that $L_1f_0$ is a
$C^\alpha$ function which vanishes for $T=0$; it is therefore
$O(d^\alpha)$ as desired.
\end{proof}

We are now ready to prove theorem \ref{th:L0}. Let $a$ be a
constant, and $f=G[ak]$. We therefore have $L'_0f=ak$, and, for
$T=0$, $f=-\frac 12 ak$. Since $L_1f\in C^\alpha$, and $L_1f$ and
$Df$ both vanish for $T=0$, it follows that, for $T=0$,
\[
Lf-k=(L'_0-(n-2)(D+2)+L_1)f-k=[a+(n-2)a-1]k.
\]
Taking $a=1/(n-1)$, we find that $f$ has the announced properties.
\end{proof}



\subsubsection{Solution of $Lw_0=g$}
\label{sec:consL}

Let us now consider a function $g$ of class
$C^\alpha(\overline\Omega_\delta)$.

Recall that there is a positive $r_0<\delta$ such that any ball of
radius $r_0$, centered at a point of the boundary, is contained in
a domain in which we have a system of coordinates of the type
$(Y,T)$. Let us cover (a neighborhood of) $\pa\Omega$ by a finite
number of balls $(V_\lambda)_{\lambda\in\Lambda}$ of radius
$r_1<r_0$ and centers on $\pa\Omega$, and consider the balls
$(U_\lambda)_{\lambda\in\Lambda}$ of radius $r_0$ with the same
centers. Thus, we may assume that every $U_\lambda$ is associated
with a coordinate system $(Y_\lambda,T_\lambda)$ of the type
considered in section \ref{sec:prel}; taking $r_1$ smaller if
necessary, we may also assume that $\overline{V}_\lambda\subset
Q_\lambda\subset U_\lambda$, where $Q_\lambda$ has the form
\[Q_\lambda:=\{(Y_{\lambda,1,\dots,}Y_{\lambda,n-1},T_\lambda) : 0\leq Y_{\lambda,
j}\leq\theta\text{ for every $j$, and }0<T_\lambda<\theta\}.
\]
Consider a smooth partition of unity $(\varphi_\lambda)$ and
smooth functions $(\Phi_\lambda)$, such that
\begin{enumerate}
  \item $\sum_{\lambda\in\Lambda}\varphi_\lambda=1$ near $\pa\Omega$;
  \item supp\,$\varphi_\lambda\subset V_\lambda$;
  \item supp\,$\Phi_\lambda\subset U_\lambda\cap\{T<\theta\}$;
  \item $\Phi_\lambda=1$ on $V_\lambda$.
\end{enumerate}
In particular, $\Phi_\lambda\varphi_\lambda=\varphi_\lambda$.

The function $g\varphi_\lambda$ is of class $C^\alpha(\overline
Q_\lambda)$; it may be extended by successive reflections to an
element of $C^\alpha_{\text{per}}$, with period $2\theta$ in the
$Y_\lambda$ variables; this extension will be denoted by the same
symbol for simplicity.

Let us apply theorem \ref{th:L0}, and consider, for every
$\lambda$, the function $w_\lambda:=G[g\varphi_\lambda/(n-1)]$. We
have
\[
Lw_\lambda=g\varphi_\lambda+R_\lambda,
\]
in $U_\lambda\cap\{T<\theta\}$, where $R_\lambda$ is H\"older
continuous for $T\leq \theta$, and vanishes on $\pa\Omega$; as a
consequence, $R_\lambda=O(d^\alpha)$.

The function $\Phi_\lambda w_\lambda$ is compactly supported in
$U_\lambda$, and may be extended, by zero, to all of $\Omega$; it
is of class $C_\#^{2+\alpha}(\overline\Omega)$. We may therefore
consider
\[ w_1:=\sum_{\lambda\in\Lambda}\Phi_\lambda w_\lambda,\]
which is supported near $\pa\Omega$. Now, near $\pa\Omega$,
\begin{eqnarray*}
\sum_{\lambda}L(\Phi_\lambda
w_\lambda)&=&\sum_{\lambda}\Phi_\lambda L(w_\lambda)+
  2d^2\nabla\Phi_\lambda\cdot\nabla w_\lambda\\
& &\qquad\qquad  \mbox{}+
  d^2w_\lambda\Delta\Phi_\lambda
  +(4-n)w_\lambda d\nabla d\cdot\nabla\Phi_\lambda  \\
  &=& \sum_{\lambda}g\Phi_\lambda\varphi_\lambda+R'_\lambda
  = g+f,
\end{eqnarray*}
where $f=\sum_{\lambda}R'_\lambda$ has the same properties as
$R_\lambda$. It therefore suffices to solve $Lw_2=f$ when $f$ is a
is H\"older continuous function which vanishes on the boundary.
\begin{lemma}\label{lem:4}
For any $f\in C^\alpha(\overline\Omega)$, there is, for $\delta$
small enough, an element $w_2\in
C_\#^{2+\alpha}(\overline\Omega_\delta)$ such that
\[ Lw_2=f\text{ and }w_2=O(d^\alpha)\text{ near }\pa\Omega.\]
\end{lemma}
\begin{proof}
Consider the solution $w_\ep$ of the Dirichlet problem $Lw_\ep=f$
on a domain of the form $\{\ep<d(x)<\delta\}$, with zero boundary
data. As before, $\delta$ is taken small enough to ensure that
$d\in C^{2+\alpha}(\overline\Omega_\delta)$. Schauder theory gives
$w_\ep\in C^{2+\alpha}(\{\ep\leq d(x)\leq\delta\})$. By
assumption, $|f|\leq a d^\alpha$ for some constant $a$. Let
$A>(\alpha+2)(n-1-\alpha)$. Since
\[
-L(d^\alpha)=d^\alpha[(\alpha+2)(n-1-\alpha)-\alpha d\Delta d],
\]
$A d(x)^\alpha$ is a super-solution if $\delta$ is small, and the
maximum principle gives us a uniform bound on $w_\ep/d^\alpha$. By
interior regularity, we obtain that, for a sequence $\ep_n\to 0$,
the $w_{\ep_n}$ converge in $C^2$, in every compact away from the
boundary, to a solution $w_2$ of $Lw_2=f$ with $w_2=O(d^\alpha)$.
Since the right-hand side $f$ is also $O(d^\alpha)$, we obtain, by
the ``type (I)'' theorem \ref{th:FIb}, that $w_2$  of class
$C_\#^{1+\alpha}(\overline\Omega_\delta)$. Theorem \ref{th:FIIa}
now ensures that $w_2$ is in fact of class
$C_\#^{2+\alpha}(\overline\Omega_\delta)$, QED.
\end{proof}
It now suffices to take $g=-2\Delta d$ and let
\[ w_0=w_1-w_2.
\]
By construction,  $Lw_0+2\Delta d=0$ near the boundary, and $w_0$
is of class $C_\#^{2+\alpha}(\overline\Omega_\delta)$ if $\delta$
is small. In addition, we know from theorem \ref{th:L0} that
$w_1\,\rest{\pa\Omega}=(2\Delta d)/(2n-2)$, which is equal to $-H$
on $\pa\Omega$. Lemma \ref{lem:4} gives us $w_2=O(d^\alpha)$. We
conclude that $w_0\,\rest{\pa\Omega}=-H$ on the boundary.

This completes the proof of theorem \ref{th:III}.


\subsubsection{Second comparison argument}
\label{sec:comp2}

At this stage, we have the following information, where
$\Omega_\delta=\{x : 0<d(x)<\delta\}$, for $\delta$ small enough:
\begin{enumerate}
  \item $w$ and $d\nabla w$ are bounded near $\pa\Omega$;
  \item $w=w_0+\tilde w$, where $L\tilde w=M_w(w)=O(d)$, and
  \item $w_0$
  is of class $C_\#^{2+\alpha}(\overline\Omega_\delta)$ for
  $\delta$ small enough.
\end{enumerate}
We wish to estimate $\tilde w$. Write $|M_w(w)|\leq cd$, where $c$
is constant.

For any constant $A>0$, define
\[ w_A:=w_0+Ad.
\]
Since $L(d)=3(2-n)d+d^2\Delta d$, we have
\[ L(w_A - w)=L(Ad-\tilde w)\leq Ad[3(2-n)+d\Delta d]+cd.
\]
Choose $\delta$ so that, say, $2(2-n)-d\Delta d\leq 0$ for $d\leq
\delta$. Then, choose $A$ so large that (i) $w_0+A\delta\geq w$
for $d=\delta$, and (ii) $(2-n)A+c\leq 0$. We then have
\[  L(w_A - w)\leq 0\text{ in }\Omega_\delta\text{ and }
      w_A - w \geq 0\text{ for }d=\delta.
\]
Next, choose $\delta$ and a constant $B$ such that
$nB+(2+Bd)\Delta d\geq 0$ on $\Omega_\delta$. We have, by direct
computation,
\[ L(d^{-2}+Bd^{-1})=-(nB+2\Delta d)d^{-1}-B\Delta d\leq 0
\]
on $\Omega_\delta$. Therefore, for any $\ep>0$,
$z_\ep:=\ep[d^{-2}+Bd^{-1}]+w_A - w$ satisfies $L z_\ep\leq 0$,
and the maximum principle ensures that $z_\ep$ has no negative
minimum in $\Omega_\delta$. Now, $z_\ep$ tends to $+\infty$ as
$d\to 0$. Therefore, $z_\ep$ is bounded below by the least value
of its negative part restricted to $d=\delta$. In other words, for
$d\leq \delta$, we have, since $w_A - w \geq 0\text{ for
}d=\delta$,
\[ w_A - w + \ep[d^{-2}+Bd^{-1}]\geq
\ep\min(\delta^{-2}+B\delta^{-1},0).
\]
Letting $\ep\to 0$, we obtain
$w_A - w\geq 0\text{ in }\Omega_\delta$.

Similarly, for suitable $\delta$ and $A$,
$w - w_{-A}\geq 0\text{ in }\Omega_\delta$.

We now know that $w$ lies between $w_0+Ad$ and $w_0-Ad$ near
$\pa\Omega$, hence $|w-w_0|=O(d)$, QED.


\end{document}